\documentclass[12pt,a4paper]{amsart}

\usepackage[a4paper, total={6in, 23.7cm}]{geometry}
\usepackage{amsmath,amsthm,amsfonts,amssymb,mathrsfs,amsbsy}
\usepackage{graphicx}
\usepackage{color}
\usepackage[noadjust]{cite}
\usepackage{marginnote}
\usepackage{xifthen}
\usepackage{mathtools}
\usepackage{dsfont}
\usepackage{bbm}
\usepackage{appendix}
\usepackage{delarray}
\usepackage{enumerate}
\usepackage{enumitem}
\usepackage{soul}
\usepackage{mathabx}
\usepackage{xfrac}
\usepackage[textwidth=1.8cm]{todonotes}
\usepackage{xcolor}
\usepackage{tikz}
\usepackage[hypertexnames=false]{hyperref} 

\hypersetup{
    colorlinks=true,
    linkcolor=blue,
    filecolor=magenta,      
    urlcolor=cyan,
    pdftitle={Overleaf Example},
    pdfpagemode=FullScreen,
    }

\mathtoolsset{showonlyrefs=true}

\binoppenalty=\maxdimen
\relpenalty=\maxdimen

\newtheorem{theorem}{Theorem}
\newtheorem{cor}[theorem]{Corollary}

\newtheorem{definition}[theorem]{Definition}

\newtheorem{lemma}[theorem]{Lemma}
\newtheorem{proposition}[theorem]{Proposition}

\numberwithin{equation}{section}
\numberwithin{theorem}{section}
\numberwithin{figure}{section}

\theoremstyle{remark}
\newtheorem{remark}[theorem]{Remark}
\newtheorem{assumption}[theorem]{Assumption}

\def\R{{\mathbb{R}}}
\def\C{{\mathbb{C}}}
\def\N{{\mathbb{N}}}
\def\X{\mathcal{X}}
\def\A{\mathcal{A}}
\def\T{\mathcal{T}}
\def\D{\mathcal{D}}
\def\Re{{\rm Re}\,}
\def\su{D_1}
\def\sv{D_2}
\def\RL{\prescript{{RL}}{}}

\usepackage{marginnote}

\newcommand{\hidden}[1]{}

\title[Fractional reaction diffusion systems]{Tools for stability analysis of fractional reaction diffusion systems}

\author[S. Ahmad]{Sofwah Ahmad} 
\address[S. Ahmad]{Department of Mathematics, College of Computing and Mathematical Sciences, Khalifa University of Science and Technology, P.O. Box 127788, Abu Dhabi, UAE\\ 
\href{https://orcid.org/0000-0001-7641-7759}{orcid.org/0000-0001-7641-7759}}
\email{100059797@ku.ac.ae, alaydrus.sofwah@gmail.com}

\author[S. Cygan]{Szymon Cygan} 
\address[S. Cygan]{	Instytut Matematyczny, Uniwersytet Wroc\l{}awski, pl. Grunwaldzki 2/4, \hbox{50-384} Wroc\l{}aw, Poland \\ 
\href{https://orcid.org/0000-0002-8601-829X}{orcid.org/0000-0002-8601-829X}}
\email{szymon.cygan@math.uni.wroc.pl}
\urladdr {http://scygan.math.uni.wroc.pl}

\author[G. Karch]{Grzegorz Karch} 
\address[G. Karch]{	
Instytut Matematyczny, Uniwersytet Wroc\l{}awski, pl. Grunwaldzki 2/4, \hbox{50-384} Wroc\l{}aw, Poland \\ 
\href{https://orcid.org/0000-0001-9390-5578}{orcid.org/0000-0001-9390-5578}}
\email{grzegorz.karch@math.uni.wroc.pl}
\urladdr {http://karch.math.uni.wroc.pl}

\date{\today}

\subjclass{35R11, 35B35, 35K57, 35B40, 34G20}
\keywords{linearization principle, fractional reaction-diffusion systems, stability, Turing instability}

\begin{document}

    \begin{abstract}
The linearization principle states that the stability (or instability) of solutions to a suitable linearization of a nonlinear problem implies the stability (or instability) of solutions to the original nonlinear problem. In this work, we prove this principle for 
solutions of abstract fractional reaction-diffusion equations with a fractional derivative in time of order $\alpha\in (0,1)$. Then, we apply these results to particular fractional reaction-diffusion equations, obtaining, for example, the counterpart of the classical Turing instability in the case of fractional equations.
    \end{abstract}

    \maketitle

\tableofcontents
\newpage
\section{Introduction}
\subsection{Revisiting the Malthusian growth model}
Denote by $u_n$ the total population (the number of species) at time $n\in\mathbb{N}$. The classical discrete  Malthusian growth model states that
    \begin{align}
        \label{eq:Malth}
        u_n = u_{n-1} + r u_{n-1},
    \end{align}
where $r\in \mathbb{R}$ is a given growth rate. 
Assuming that $u_n = u(nh)$ for some differentiable function $u= u(t)$, with arbitrary small $h>0$, and for the growth rate $r = r_0 h$
with fixed $r_0\in\R$,
after passing to the limit $h\to 0$,
we obtain the continuous version of the Malthusian growth model
    \begin{align}\label{eq:Malth:2}
       \frac{d}{dt} u = r_0 u \quad \text{with the solution} \quad u(t)= u(0) e^{r_0t}.
    \end{align}

Now, we rewrite the Malthusian model \eqref{eq:Malth} by direct calculation in the following form
    \begin{align} \label{eq:Malth:1}
        u_n &= u_0 + \sum_{k=0}^{n-1} r u_k = u_0 + \sum_{k=1}^{n} r u_{k-1} \\
        &= u_0 + \sum_{k=1}^{n}  \left(r u_0 + \sum_{j=1}^{k-1} r\left( u_j - u_{j-1}\right)\right),
    \end{align}
    with the convention that for $n=1$, we set $\sum_{k=1}^n \sum_{j=1}^{k-1} r(u_{j}-u_{j-1})=0.$
Notice that we have the following components in formula \eqref{eq:Malth:1}
    \begin{itemize}
        \item $u_0$ -- the initial number of the species,
        \item $r u_0$ -- the number of the descendants (offspring) of the initial population born at every unit of time, 
        \item $u_j - u_{j-1}$ -- number of descendants (offspring) born at $j$-th unit of time (\textit{called $j$-th generation})
        \item $r (u_{j} - u_{j-1})$ -- the number of descendants of the $j$-th generation born after every unit of time,
        \item $ru_0 + \sum_{j=1}^{k-1} r\left( u_j - u_{j-1}\right)$ total number of descendants (offspring) born at $k$-th unit of time.
    \end{itemize}

In equation \eqref{eq:Malth:1},  all generations have the same growth rate $r\in\R$. 
We generalize this model by taking into account that the growth rate of each individual may depend on its age.
More precisely, we introduce numbers $r_k\in \mathbb{R}$ (with $k\in \mathbb{N}$) which are the growth rates of a generation after $k$ units of time (years, days, \textit{etc.}) and propose the following model
    \begin{equation}
        \begin{aligned}
            u_n &= u_0 + \sum_{k=1}^{n} \left(r_k u_0  + \sum_{j=1}^{k-1} r_{k-j}\left(u_j - u_{j-1}\right) \right),
            \label{eq:ModelGen}
        \end{aligned}
    \end{equation}
    where, as in equation \eqref{eq:Malth:1}, for $n=1$, we choose $\sum_{k=1}^n \sum_{j=1}^{k-1} r_{k-j}(u_{j}-u_{j-1})=0$.
Formula \eqref{eq:ModelGen} means that the following numbers contribute to  
the total population $u_n$:
    \begin{itemize}
        \item $u_0$ -- the number of the species at the beginning,
        \item $r_k u_0$ -- the number of the descendants of the initial population born at $k$-th unit of time, 
        \item $r_{k-j} (u_{j} - u_{j-1})$ -- the number of descendants of the $j$-th generation born at $k$-th unit of time,
        \item $r_ku_0 + \sum_{j=1}^{k-1} r_{k-j}\left( u_j - u_{j-1}\right)$ total number of descendants born at $k$-th unit of time.
    \end{itemize}

Changing the order of summation in equation \eqref{eq:ModelGen} we obtain
    \begin{equation}
        \begin{aligned}
            u_n &= u_0 + \left(\sum_{k=1}^{n} r_k u_0 \right) + \sum_{j=1}^{n-1} \left(\sum_{k=j+1}^{n} r_{k-j}\left(u_j - u_{j-1}\right) \right) \\
            &=  u_0 + \left(\sum_{k=1}^{n} r_k u_0 \right) + \sum_{j=1}^{n-1}\left( \left(u_j - u_{j-1}\right)  \sum_{k=1}^{n-j} r_{k} \right).  
            \label{eq:ModelGen2}
        \end{aligned}
    \end{equation}

Denoting $S_0=0$ and $S_k = \sum_{j=1}^k r_j$ for $k\ge 1$ we rewrite  formula \eqref{eq:ModelGen2} in the form
    \begin{equation}
        \begin{split}
        \label{eq:ModelSum}
        u_n &= u_0 + u_0 S_{n} + \sum_{k=1}^{n-1} (u_k - u_{k-1}) S_{n-k}\\& = u_0 + \sum_{k=0}^{n-1} \left(S_{n-k} - S_{n-k-1}\right)u_k.
     \end{split}
    \end{equation}
Note that $S_{n-k} - S_{n-k-1}=r_{n-k}$, and equation \eqref{eq:ModelSum} takes the form
\begin{equation}
    u_n = u_0 + \sum_{k=0}^{n-1} r_{n-k}u_k.
\end{equation}
Thus, it would seem that we could obtain this equation directly by replacing $r$ with $r_{n-k}$ in the first equation in \eqref{eq:Malth:1}.
However, to give a biological motivation for the model obtained in this way, we should still go through 
formula \eqref{eq:ModelGen} and 
the calculations in \eqref{eq:ModelGen2}
involving the contribution of $j$-th generations to the total population.

In order to obtain a continuous counterpart of model \eqref{eq:ModelSum}, we assume that $u(t)$ is a continuous function and $S(t)$ is continuously differentiable. 
For $t>0$, we introduce the numbers 
    \begin{align}
            h=t/n \qquad \text{and}\qquad 
            t_k = kh \quad \text{for} \quad k\in\lbrace 1, \ldots, n-1\rbrace.
    \end{align}
Thus, for $u_k = u(t_k)$ and $S_k = S(t_k)$, equation \eqref{eq:ModelSum} takes the form
    \begin{align}
        \label{eq:ModelSumCont}
        u(t) &= u_0 + \sum_{k=1}^{n-1} \big(S(t - t_k) - S(t - t_k - h)\big)u_k\\ & = u_0 + \sum_{k=1}^{n-1} \frac{S(t - t_k) - S(t - t_k - h)}{h} \, u_k h, 
    \end{align}
where the right-hand side is the Riemann approximation of a certain integral. Indeed, by passing to the limit with $h\to 0$,
we obtain the continuous counterpart  of model \eqref{eq:ModelSum} (hence also of model \eqref{eq:ModelGen})
    \begin{align}
        u(t) = u(0) + \int_0^t S'(t-s) u(s) \, ds.
    \end{align}

Notice that by the definition of $S_k$ in formula \eqref{eq:ModelSum}, the quantity 
$S'(\tau)$ can be interpreted as a growth rate of $\tau$-th generation ({\it i.e.} the population which was born exactly $\tau$-time ago).  
Here, we discuss some particular versions of this function. 
If $S(\tau) = r_0 \tau$, we obtain the equation
    \begin{align}
        u(t) = u(0) + r_0 \int_0^t u(s)\,  d s,
    \end{align}
which is the integral formulation of the classical model \eqref{eq:Malth:2}. 
Similarly, the step-like function \begin{align}
        S(\tau) =  \begin{cases}
            0, & \text{for} \quad \tau \in [0, \ t^*), \\
            r_0 (\tau-t^*), & \text{for} \quad \tau \in [t^*, \ \infty),
        \end{cases}
    \end{align}
leads the equation  the integral formulation of the time delay equation
    \begin{align}
        u(t) = u(0) + r_0\int_0^{t-t^*} u(s)\,  ds,
    \end{align}
corresponding the Cauchy problem for the time delay equation $u'(t)=r_0 u(t-t^*)$.

In this work,  we choose the particular function 
    \begin{equation}
        S(\tau) = r_0  \frac{\tau^\alpha}{\alpha \Gamma(\alpha)} \quad \text{with}\quad  \alpha \in (0,1) \quad\text{and}\quad  r_0\in\R,
    \end{equation}
to obtain the Volterra equation
    \begin{align}\label{Volterra:0}
        u(t) = u(0) +  \frac{r_0}{\Gamma(\alpha)}\int_0^t (t-s)^{\alpha-1} u(s) \, ds , 
    \end{align}
which appears to have an equivalent formulation as a fractional differential equation. 

\subsection{Fractional derivatives} 
In order to write integral equation \eqref{Volterra:0} as a fractional differential equation, we briefly recall some well-known facts   (see, {\it e.g.}, \cite{KilbasSrivastavaTrujillo}) 
from  the theory of fractional-in-time derivatives.

Let $\alpha \in (0, 1]$. In the following,  $u=u(t)$ represents an arbitrary $C^1$-function defined for $t\ge 0$.
The Riemann-Liouville fractional integral 
    \begin{align}\label{fractional integral}
        J_t^{\alpha} u(t) \equiv
        \frac{1}{\Gamma(\alpha)}\int_{0}^{t} (t-\tau)^{\alpha-1} u(\tau)\, d\tau,
    \end{align}
    and the Caputo fractional derivative 
    \begin{equation}
        \partial^\alpha_t u(t) = J_t^{1-\alpha} \left(\frac{d}{dt} u\right)(t)=\frac{1}{\Gamma(1-\alpha)} \int_{0}^{t} (t-\tau)^{- \alpha} u'(\tau)\, d\tau,
    \end{equation}
are related by the formula
     \begin{equation}
         \partial_t^{\alpha} J_t^{\alpha} u(t) =u(t).
     \end{equation}
Thus, applying the fractional derivative $ \partial_t^{\alpha}$ to equation \eqref{Volterra:0}  written in the form $u(t)=u(0)+r_0J_t^\alpha (u(t))$ 
we obtain the fractional differential equation 
    \begin{align}
        \partial_t^\alpha u(t) = r_0u(t),
    \end{align}
which we discuss in the next section.


\section{Fractional linear Cauchy problem} \label{sec:linear}
Here, we recall results on linear equations with Caputo fractional-in-time derivatives of order $\alpha\in (0,1)$,
which will be needed in our considerations.

\subsection{Fractional differential equations} The simple fractional differential equation
        \begin{align}
            \partial^\alpha_t u &= \lambda u,
            \label{eq: FDE simple:0}
        \end{align}
 with $\alpha \in (0,1]$, a parameter $\lambda \in \mathbb{C}$, and  supplemented with an initial datum $u(0)=u_0$ has 
the solution in the following form 
        \begin{align}
            u(t) = u_0 \ E_{\alpha}(\lambda t^\alpha), \quad\text{where}\quad E_{\alpha}(z)=E_{\alpha,1}(z)
        \label{eq: sol D^a u=ru:0}
        \end{align}
         belongs to the  family of the two parameters Mittag-Leffler functions (see, {\it e.g.},~\cite{diethelm2010analysis,GKMR20})
    given by the formula
        \begin{equation}\label{Mittag-Leffler}
            E_{\alpha,\beta}(z) = \sum_{k=0}^\infty \frac{z^k}{\Gamma(\alpha k +\beta)}, \qquad \text{for}\quad \alpha,\beta,z \in \mathbb{C}\;\;\text{with} \;\; \Re \alpha>0.
        \end{equation}
Here, we emphasize that the functions $E_\alpha$ and $E_{\alpha,\alpha}$ with $\alpha\in (0,1]$ appear when solving 
 the Cauchy problem 
\begin{align}
            \partial^\alpha_t u &= \lambda u+f(t),\\
            u(0)&=u_0,
            \label{eq:inhom}
        \end{align}
with 
$\alpha \in (0,1]$, 
$\lambda \in \C$, and  $f\in C\big([0,\infty)\big)$,
which has the explicit solution 
    \begin{equation}\label{Duh:0}
        u(t) = E_\alpha(\lambda t^\alpha) u_0 + \int_0^t (t-s)^{\alpha-1} E_{\alpha,\alpha}\big(\lambda (t-s)^\alpha\big) f(s) \, d s.
    \end{equation}
Moreover, both functions are related by the formula
$$ E_\alpha (\lambda t^\alpha) = J^{1-\alpha}_t\big( t^{\alpha - 1} E_{\alpha, \alpha} (\lambda t^\alpha)\big) $$
and they
satisfy the following
initial value problems (see, {\it e.g.}, \cite[Thm.~7.2 and Rem.~7.1]{diethelm2010analysis} or \cite[Sec.~7.2.1]{GKMR20})
    \begin{equation}\label{Ea:Cauchy}
        \begin{split}
            \partial_t^\alpha E_\alpha (\lambda t^\alpha) &= \lambda E_\alpha (\lambda  t^\alpha),\\
            E_\alpha(0)&=1
        \end{split}
    \end{equation}
        and
    \begin{equation}\label{Eaa:Cauchy}
        \begin{split}
            \RL\partial_t^\alpha \left( t^{\alpha - 1} E_{\alpha, \alpha} (\lambda t^\alpha) \right) &= \lambda t^{\alpha - 1} E_{\alpha, \alpha} (\lambda t^\alpha),
            \\
            \lim_{t\to 0} J^{1-\alpha}_t\left( t^{\alpha - 1} E_{\alpha, \alpha} (\lambda t^\alpha)\right)&=1,
        \end{split}
    \end{equation}
where $\RL\partial_t^\alpha$ denotes the Riemann-Liouville fractional derivative of order $\alpha\in (0,1)$ given by the formula
    \begin{equation}\label{RL:derivative}
        \RL\partial^\alpha_t u(t) = \frac{d}{dt} J_t^{1-\alpha}  u(t)=\frac{1}{\Gamma(1-\alpha)} 
        \frac{d}{dt} \int_{0}^{t} (t-\tau)^{- \alpha} u(\tau) d\tau.
    \end{equation}

    The Mittag-Leffler functions can also be represented as the Laplace transforms of the classical Wright function (see, for example,
    \cite[Ch.~F.2]{GKMR20} for other properties of the Wright function and for additional references).
    \begin{lemma} 
    \label{lem:E:E}
        For every $\alpha\in (0,1)$ and $z\in \C$
            \begin{equation} \label{subord:ML}
                E_\alpha(z) = \int_0^\infty \Psi_\alpha (s) e^{zs}\, ds \quad  \text{and} \quad
                E_{\alpha,\alpha} (z) = \int_0^\infty \alpha s \Psi_\alpha (s) e^{zs}\, ds,
            \end{equation}
        with the function of the Wright-type
            \begin{equation}\label{Wright}
                \Psi_\alpha(z)= \sum_{n=1}^\infty \frac{(-z)^n}{n!\Gamma(-\alpha n+1-\alpha)},
            \end{equation}
        which has the following properties: 
            \begin{equation*}
                \Psi_\alpha(s)\geq 0\quad \text{for} \quad s>0 \quad \text{and}\quad \int_0^\infty \Psi_\alpha (s) \, ds=1.
            \end{equation*}
    \end{lemma}

Finally, we recall the well-known asymptotic properties of the considered Mittag-Leffler functions.

\begin{lemma}\label{lem:ML}
Let $\alpha \in (0,1]$ and $\lambda \in \mathbb{C}$.
There exist positive real numbers $m(\alpha, \lambda)>0$ such that the following inequalities hold for all $t>0$. 
    \begin{enumerate}
        \item If $|\arg \lambda| >\alpha\pi/2 $, then
            \begin{align}\label{Ealpha:Ealphaalphs:estimates}
                \left|E_\alpha(\lambda t^\alpha)\right|&\le m(\alpha, \lambda) \min\{t^{-\alpha},1\},
                \\
                \left|t^{\alpha-1}E_{\alpha,\alpha}(\lambda t^\alpha)\right|&\le m(\alpha, \lambda) \min\{t^{-\alpha-1},t^{\alpha-1}\}.
            \end{align}
           
        \item If $|\arg \lambda| <\alpha\pi/2$, then 
            \begin{align}
                \left|E_\alpha(\lambda t^\alpha)- \frac{1}{\alpha} \exp(\lambda^{1/\alpha}t) \right|&\le m(\alpha, \lambda) \min\{t^{-\alpha},1\},
                \label{Ea:growth}
                \\
                \left|t^{\alpha-1}E_{\alpha,\alpha}(\lambda t^\alpha)- \frac{1}{\alpha} \lambda^{1/\alpha-1}\exp(\lambda^{1/\alpha}t) \right|&\le m(\alpha, \lambda) \min\{t^{-\alpha-1},t^{\alpha-1}\}.
            \end{align}
        Moreover, there exist numbers $T_1=T_1(\alpha,\lambda)>0$ and $C=C(\alpha,\lambda)>0$ such that
         \begin{equation}\label{Ea:lower}
                C(\alpha,\lambda) \exp\big((\Re \lambda)^{1/\alpha}t\big) \le\left|E_\alpha(\lambda t^\alpha)\right| \qquad\text{for all} \quad
                t\geq T_1.
            \end{equation}
    \end{enumerate}
\end{lemma}
\begin{proof}
 For $\alpha=1$, the estimates are true because $E_1(z)=E_{1,1}(z)=\exp(z)$.
 To prove them for $\alpha \in (0,1)$, one should use the following asymptotic expansions of the Mittag-Leffler function (see, {\it e.g.}, \cite[Sec.~4.7]{GKMR20})
 which hold true for all $p \in \mathbb{N}$:
        \begin{itemize}
            \item if $ z \in \mathbb{C}$ with $|\arg(z)|<\alpha\pi/2$ then
                \begin{align}
                    E_{\alpha, \beta}(z)&=\frac{1}{\alpha} z^{(1-\beta) / \alpha} \exp \left(z^{1 / \alpha}\right)-\sum_{k=1}^p \frac{z^{-k}}{\Gamma(\beta-\alpha k)}+O\left(|z|^{-p-1}\right) \quad\text{when} \quad |z|\to\infty;
                \label{eq: asymptotics Ealphabeta unstable:0}
                \end{align}
                
            \item if $ z \in \mathbb{C}$ with $|\arg(z)|>\alpha\pi/2$  then
                \begin{equation}
                    E_{\alpha, \beta}(z)=-\sum_{k=1}^p \frac{z^{-k}}{\Gamma(\beta-\alpha k)}+O\left(|z|^{-p-1}\right)
                    \quad\text{when} \quad |z|\to\infty.
                \label{eq: asymptotics Ealphabeta stable:0}
                \end{equation}
        \end{itemize}
        Here, $z^\gamma$ for $z\in\C$ and $\gamma\in\R\setminus\N$ denotes the principal branch of the corresponding multivalued function.
We refer, for example, to \cite[Lemma A1]{DiethelmSiegmundTuan2017} for more detailed explanations of Parts 1 and 2.
 Inequality \eqref{Ea:lower} immediately follows from the estimates in Part (2), because $\Re\lambda>0$.
\end{proof}        
      
\begin{remark}\label{rem:stab:FDE}
    Lemma \ref{lem:ML} implies immediately that 
if $|\arg \lambda| >\alpha\pi/2 $, then the solution $u(t)\equiv 0$  of the linear homogeneous equation 
$\partial_t^\alpha u=\lambda u$
is asymptotically stable and is unstable if $|\arg \lambda| <\alpha\pi/2$. Moreover, for $|\arg \lambda| >\alpha\pi/2 $ and for $f\in ([0,\infty))$ satisfying $\lim_{t\to\infty} f(t)=0$, the solution of the linear inhomogeneous equation 
$\partial_t^\alpha u=\lambda u+f$
given by formula 
\eqref{Duh:0} satisfies $\lim_{t\to\infty} u(t)=0$.
This can be shown by combining estimates \eqref{Ealpha:Ealphaalphs:estimates} with the reasoning in Lemma \ref{lem:integral}, 
below (see, for example,  \cite[Lemma~A.4]{DiethelmSiegmundTuan2017} for detailed calculations).
\end{remark}

\subsection{Abstract linear fractional Cauchy problem.}
Now, we consider a general
      densely defined, closed (possibly unbounded) linear operator 
       $(\A, D(\A))$
      on a Banach space $(\X, \|\cdot\|)$ and the linear fractional differential equation
 \begin{equation}
     \partial_t^\alpha u=\A u \qquad\text{with} \quad \alpha\in (0,1).
 \end{equation}     
 
      \begin{assumption}\label{ass:A}
For a certain Banach space $(\X, \|\cdot\|)$, we assume  that the linear operator 
$
\A : D(\A )\subset \X \mapsto \X
$
 generates a strongly continuous  semigroup of linear operators $\{T(t)\}_{t\geq 0}$ on $\X$.
 Moreover, we assume that the Spectral Mapping Theorem holds true, namely, 
 \begin{equation}\label{spectra:mapping}
     \sigma(T(t))\setminus \{0\}=e^{t\sigma(\A )} \quad \text{for all} \quad t\ge 0.
 \end{equation}
\end{assumption}

It is well known that the semigroup $\{T(t)\}_{t\geq 0}$ is exponentially bounded (see {\it e.g} \cite[Ch. 1, Proposition 5.55]{Engel-Nagel}), namely,  there exist constants $M\geq 1$ and $\omega \in \R$ such that, for every $u_0\in \X$,  the following estimate holds true
    \begin{equation}
    \label{Tt:exp}
        \|T(t) u_0\|\leq Me^{\omega t}\|u_0\| \qquad \text{for all} \quad t\geq 0.
    \end{equation}
    Recall that the Spectral Mapping Theorem holds true for sufficiently regular semigroups \cite[Corollary 3.12]{Engel-Nagel}:
for example, for
uniformly continuous semigroups and
analytic semigroups.

In this general setting, we discuss  the Cauchy problem for the abstract linear fractional differential equation with $\alpha \in (0,1]$
and with 
  $f\in C\big([0,\infty),\X\big)$
 \begin{align}
            \partial^\alpha_t u &= \A u+f(t),\\
            u(0)&=u_0,
            \label{eq: FDE simple:1}
        \end{align}
which has a unique  mild solution of the following form 
    \begin{equation}\label{duh:lin}
        u(t) = S_\alpha(t)u_0 +\int_0^t (t-s)^{\alpha-1}P_\alpha (t-s) f(s) \, ds.
    \end{equation}
    Here, 
 the {\it resolvent family} $\{S_\alpha(t)\}_{t\geq 0}$
  is a strongly continuous family of bounded and
linear operators defined on $\X$, such that $S_\alpha(0)=I$,
which commute with $\A $ and satisfy the resolvent
equation
\begin{equation}
        S_\alpha (t)u_0 = u_0 + \int_0^t \frac{(t-s)^{\alpha-1}}{\Gamma(\alpha)} \A  S_\alpha  (s)u_0\, ds, 
    \end{equation}
    for all  $u_0\in D(\A )$ and  $t \geq 0$.
The second family $\{P_\alpha(t)\}_{t> 0}$ in equation \eqref{duh:lin} is called the {\it integral resolvent} and it consists of a strongly continuous family of bounded and
linear operators, which commute with $\A $ and satisfy the equation
$$
t^{\alpha-1}P_\alpha (t)u_0= \frac{t^{\alpha-1}}{\Gamma(\alpha)} u_0  +
\int_0^t \frac{(t-s)^{\alpha-1}}{\Gamma(\alpha)} s^{\alpha-1}\A  P_\alpha  (s)u_0\, ds 
$$
for all  $u_0\in D(\A )$ and  $t > 0$.
Here, we also recall the well-known relation for these resolvent families which is obtained directly from their definitions and 
which hold true for all $u_0\in D(\A )$ and all $t>0$:
 \begin{equation} \label{Salpha:Cauchy:problem}
      S_\alpha(t)u_0= \frac{1}{\Gamma(1-\alpha)}\int_0^t (t-s)^{-\alpha} s^{\alpha-1} P_\alpha(s)u_0\, ds.
  \end{equation}
Moreover, both resolvents are solutions of the initial value problems
\begin{align}\label{Sa:Cauchy}
    \partial^\alpha_t S_\alpha(t)u_0 &= \A S_\alpha (t)u_0, \\
    S_\alpha(0)u_0&=u_0
\end{align}
and
\begin{equation} 
    \begin{split}
        \label{Pa:Cauchy}
        \RL\partial^\alpha_t \big(t^{\alpha-1}P_\alpha(t)u_0\big) &= \A \big(t^{\alpha-1}P_\alpha (t)\big)u_0,\\
        \lim_{t\to 0} J^{1-\alpha}_t\big(t^{\alpha-1} P_\alpha(t)u_0\big)&=u_0.
    \end{split}
\end{equation}

The resolvent operators $S_\alpha(t)$ and $P_\alpha(t)$ are extensions of the Mittag-Leffler functions $E_\alpha(t)$ and $E_{\alpha,\alpha}(t)$
to Banach spaces and have analogous properties (compare, for example, the Cauchy problems \eqref{Sa:Cauchy} and \eqref{Pa:Cauchy} with those in \eqref{Ea:Cauchy} and \eqref{Eaa:Cauchy}.
These resolvents were introduced by Pr\"uss \cite{P93} who extensively studied their properties (also those mentioned above) and their connections to the abstract Cauchy and Volterra equations.
Those results have been further generalized and expanded in several other works; see, for example, \cite{L00,Bajlekova2001,KJ12,WCX12,MPZ13,KLW16,AML16,AA18} for the proofs of the relations above and for other references.

The definitions of the operators $S_\alpha(t)$ and $P_\alpha(t)$ do not require that $\A $ generates a semigroup.
However, in the case where $\A $ generates a strongly continuous semigroup of linear operators, as stated in Assumption \ref{ass:A},
the resolvent families $S_\alpha(t), P_\alpha(t):\X \to \X$ 
are given explicitly by the following subordination formulas which are due to Bazhlekova \cite{B00,Bajlekova2001} (see also \cite{KLW16,MT19} for generalizations)
\begin{equation}\label{Salpha}
    S_\alpha(t) u_0 = \int_0^\infty \Psi_\alpha (s) T(st^\alpha) u_0\, ds 
 \end{equation}
 and
 \begin{equation}\label{Palpha}
      P_\alpha(t) u_0 =   \int_0^\infty \alpha s \Psi_\alpha (s) T(st^\alpha) u_0\, ds,
\end{equation}
where $\Psi_\alpha$ is the Wright-type function \eqref{Wright} (compare with the formulas for $E_\alpha$ and $E_{\alpha,\alpha}$ in Lemma \ref{lem:E:E}).



\subsection{Linear stability}\label{sec:lin:stab}

The stability of a strongly continuous semigroup of linear operators $\{T(t)\}_{t\ge0}$, expressed by its convergence to zero as $t\rightarrow\infty$, is often connected with the properties of the spectrum of the operator $(\A , D(\A ))$ that generates the semigroup. To explain this idea, it is useful to introduce the growth bound
\begin{align}
\omega_{0}(\A )
=\inf\left\{\omega \in\mathbb{R}\; : \; \exists M_{\omega}\ge1 \quad\text{such that} \quad \|T(t)\|\leq M(\omega) e^{\omega t}
\quad \text{for all} \quad t\ge 0
\right\}.
\end{align}
From this definition, it is immediately clear that $\{T(t)\}_{t\ge0}$ is uniformly exponentially stable if and only if
$$\omega_{0}(\A )<0.$$  
Now, we recall (following, for example, the monograph \cite[Chapter IV]{Engel-Nagel}) the direct characterizations of uniform exponential stability of the semigroup in terms of its generator by using spectral theory. More precisely, if the operator and the semigroup generated by it satisfy the Spectral Mapping Theorem \eqref{spectra:mapping}, then the spectral bound
$$s(\A ) \equiv \sup\{\Re \lambda \;:\; \lambda\in\sigma(\A )\}$$
satisfies $s(\A )=\omega_0(\A )$. Consequently, such a semigroup $\{T(t)\}_{t\ge0}$ is uniformly exponentially stable if $s(\A )<0$.

In this paper, we extend this idea to the resolvent families 
and {\it the linear stability} is expressed by the following estimates (see Assumption~\ref{ass:stab:resolvent}, below)
\begin{equation} \label{SP:ass:stab:000}
               \|S_\alpha (t)v\|\le K \|v\| \min\{t^{-\alpha},1\}
               \quad\text{and}\quad    \|t^{\alpha-1}P_\alpha (t)v\|\le  K \|v\| \min\{t^{-\alpha-1},t^{\alpha-1}\}
    \end{equation}
        for  all $v\in \X$ and all $t> 0$.
In other words, in the case of linear stability, we require that $\|S_\alpha (t)v\|$ and $\|P_\alpha (t)v\|$ decay algebraically in the same way as the Mittag-Leffler functions $E_\alpha(\lambda t)$ and $E_{\alpha,\alpha}(\lambda t)$ with $|\arg \lambda| >\alpha\pi/2 $, see Lemma \ref{lem:ML}.

In the following, we discuss the conditions imposed on $\sigma(\A )$ under which these two decay estimates are true, and we begin with the following general estimate. 

\begin{lemma}[{\cite[Corollary 3.2]{Bajlekova2001}}]\label{lem:SPE}
Let $\alpha\in (0,1)$.
    For the strongly continuous semigroup $\{T(t)\}_{t\geq 0}$ 
    satisfying inequality \eqref{Tt:exp} with some $M\ge 1$ and $\omega\in \R$
    and for the corresponding resolvents given by formulae \eqref{Salpha} and \eqref{Palpha},
    we have the estimates
       \begin{align}
           \|S_\alpha (t)u_0\|\le M E_\alpha(\omega t^\alpha)\|u_0\| 
           \qquad\text{and}\qquad
           \|P_\alpha (t)u_0\|\le M E_{\alpha,\alpha}(\omega t^\alpha)\|u_0\|,
       \end{align}
    for all $t\ge 0$ and all $u_0\in \X$.
\end{lemma}
\begin{proof}
    Apply the norm $\|\cdot \|$ to the explicit expressions for $S_\alpha (t)$ and $P_\alpha (t)$ in \eqref{Salpha}-\eqref{Palpha}, then use the semigroup estimate \eqref{Tt:exp} and Lemma \ref{lem:E:E}.
\end{proof}

Consequently, the negative spectral bound of $\A$ implies linear stability in the case of resolvents.

\begin{cor}\label{cof:negative spectal bound}
Consider an operator $\A $, satisfying Assumption \ref{ass:A}, which has a negative spectral bound $s(\A )<0$.
Then the corresponding resolvent families $S_\alpha(t)$ and $P_\alpha(t)$
    satisfy estimates \eqref{SP:ass:stab:000}.
\end{cor}
\begin{proof}
    Use Lemma \ref{lem:SPE} and the inequality  
    \begin{equation}\label{T:omega:estimate}
    \|T(t)u_0\|\le M(\omega) e^{\omega t}\|u_0\| \quad\text{for all} \quad u_0\in\X, \; t\geq 0,
    \end{equation}
    which is true for every $\omega\in \big(s(\A ), 0\big)$ when the Spectral Mapping Theorem \eqref{spectra:mapping} is satisfied.
    Then, apply the estimates from Lemma \ref{lem:ML}, part 1.
\end{proof}

\begin{remark}\label{remark:linear:stability}
{\it It is not expected} that the negative spectral bound (as in Corollary \ref{cof:negative spectal bound})
is an optimal assumption to prove linear stability estimates \eqref{SP:ass:stab:000}.
There is a conjecture that it suffices to require 
\begin{equation}\label{conj:stab}
    \sigma(\A )\subset \left\{\lambda\in\C\;:\; \left|\arg(\lambda )\right| >\frac{\alpha\pi}{2}\right\}.
\end{equation}
Here, the motivation for such an expectation comes from the stability analysis of solutions to the simplest fractional differential equation
\eqref{eq: FDE simple:0} described in Lemma \ref{lem:E:E} and in Remark \ref{rem:stab:FDE}.
In this work, we prove decay estimates \eqref{SP:ass:stab:000} under assumption \eqref{conj:stab} in the study of stability of constant solutions to 
\begin{itemize}
    \item systems of fractional differential equations, when $\A $ is just a matrix - see Section \ref{sec:FDE}
    \item systems of fractional reaction-diffusion equations, when  $\A $ consists of Laplace operators perturbed by a constant coefficient matrix - see Section \ref{sec:FRD}.
\end{itemize}
\end{remark}

\begin{remark}\label{rem:Pa:from:Sa}
    Notice that the second inequality for $P_\alpha(t)$ in \eqref{SP:ass:stab:000}
implies the first estimate for $S_\alpha(t)$.
    Indeed, by equation \eqref{Salpha:Cauchy:problem}, we obtain 
    \begin{align}
         \|S_\alpha(t)v\|&\le \frac{1}{\Gamma(1-\alpha)}\int_0^t (t-s)^{-\alpha} s^{\alpha-1} \|P_\alpha(s)v\|\, ds\\
         &\le  \frac{K\|v\|}{\Gamma(1-\alpha)}\int_0^t (t-s)^{-\alpha} \min\{s^{-\alpha-1},s^{\alpha-1}\}\, ds\\
         &\le  \frac{K\|v\|}{\Gamma(1-\alpha)} C \min\{t^{-\alpha},1\},
    \end{align}
    where last estimate holds true by Lemma \ref{lem:integral}, below. 
\end{remark}

\subsection{Linear instability}\label{sec:lin:instab}
Obviously, to show that the zero solution of the abstract linear equation $\partial_tu=\A u$ is unstable, it suffices to find an eigenvalue $\lambda\in \C$ of the operator $\A $ such that $\Re\lambda>0$.
It is not very difficult to show that if $s(\A )>0$ then the zero solution is also unstable. Here, we should use the fact that each $\lambda\in \sigma(\A )$ is such that $\Re\lambda=s(\A )$ corresponds to an approximate eigenvalue and proceed, for example, as shown in \cite{SS98}.
Then, the element of $\sigma(\A )$, which satisfies $\Re\lambda=s(\A )$, together with the semigroup estimates \eqref{T:omega:estimate} with all $\omega >s(\A )>0$
are used to show the instability of the zero solution of suitable semilinear problems (see, for example,  \cite{SS98}). 
Here, we extend this idea to the resolvent families $S_\alpha(t)$ and $P_\alpha(t)$.

First, we notice that if $u_0$ is an eigenfunction of the operator $\A $ corresponding to the eigenvalue $\lambda$, then
    \begin{equation}\label{Sa:eigen}
        S_\alpha (t) u_0 = E_\alpha(\lambda t^\alpha) u_0 \quad \text{for all} \quad t\ge 0.
    \end{equation}
For the proof of formula \eqref{Sa:eigen}, it suffices to use the representation of $S_\alpha(t)$ from  \eqref{Salpha},
the following well-known fact for strongly continuous semigroups of linear operators 
$$T(t)u_0=e^{\lambda t} u_0 \quad \text{for all} \quad t\ge 0$$
(see, for example, \cite[Theorem 3.6]{Engel-Nagel}), and the representation of the Mittag-Leffler function $E_\alpha(t)$ from Lemma \ref{lem:E:E}.
Now, by the asymptotic expansion of the Mittag-Leffler function from Lemma \ref{lem:ML}, if
\begin{equation}\label{lambda:unstable}
 \lambda\neq 0 \quad\text{and\quad }   \left|\arg(\lambda )\right| <\frac{\alpha\pi}{2},
\end{equation}
then the solution \eqref{Sa:eigen} of the equation $\partial_t^\alpha u=\A u$
grows exponentially when $t\to\infty$.

The existence of the eigenvalue $\lambda\in\C$ satisfying  inequality \eqref{lambda:unstable} is not sufficient on its own
in our proof of the instability of the zero solution of semilinear problems (discussed in Section~\ref{sec:instability})
and we will also require that $\Re\lambda$ controls the exponential growth of the resolvent families (as stated in Assumption \ref{ass:instab}, below). 
The following proposition presents a non-optimal result in this direction.

\begin{proposition}\label{prop:instab}
Under Assumption \ref{ass:A}, suppose, moreover, that the operator $\A $ has 
the eigenvalue $\lambda$ such that 
    \begin{equation}
        \Re \lambda  =s(\A )>0 \qquad\text{and}\qquad \left|\arg(\lambda )\right| <\frac{\alpha\pi}{2}.
    \end{equation}
Then for every $\omega>0$ there exists $C(\omega)>0$
    such that 
    \begin{equation}\label{SaPa:instab:0000}
        \begin{split}
            \|S_\alpha(t)u_0\|
        &\le C(\omega) e^{(\Re \lambda  +\omega)^{1/\alpha}t}\|u_0\| 
        \\
          \|t^{\alpha-1} P_\alpha(t)u_0\|
        &\le C(\omega) e^{(\Re \lambda  +\omega)^{1/\alpha}t}\|u_0\|
        \end{split}
    \end{equation}
    for all $t\ge 0$ and $u_0\in\X$.
    \end{proposition}
    \begin{proof}
Since  $ \Re \lambda  =s(\A )>0$,  for every $\omega>0$ there exists $M(\omega)>0$
    such that 
    \begin{equation}
        \|T(t)u_0\|\le M(\omega) e^{(\Re \lambda  +\omega)t}\|u_0\|
        \qquad 
        \text{for all} \quad  t\ge 0 \quad \text{and} \quad  u_0\in\X
    \end{equation}
(see, {\it e.g.}, \cite[Ch.~IV, Corollary 3.12]{Engel-Nagel}).
Next, we apply the explicit formulas for the resolvent families in \eqref{Salpha} and \eqref{Palpha} combined with Lemma~\ref{lem:E:E}.
\end{proof}

\begin{remark}\label{rem:instab} 
In this work, we describe the linear instability by requiring (in Assumption \ref{ass:instab} below) the existence of $\lambda \in \sigma(\A )$ satisfying conditions \eqref{lambda:unstable} and which controls the growth of the resolvent families, as stated in \eqref{SaPa:instab:0000}.
To have such an element of the spectrum, it seems that it suffices to require that 
\begin{equation}
    \Re \lambda = \sup \left\{\Re \eta \,:\, \eta\in \sigma(\A)\quad \text{and} \quad |\arg(\eta)|< \frac{\alpha \pi}{2}  \right\};
\end{equation}
 however, we can prove this in two cases only:
\begin{itemize}
    \item when $\A $ is just a matrix - see Section \ref{sec:FDE}
    \item when $\A $ consists of Laplace operators perturbed by a constant coefficient matrix - see Section \ref{sec:FRD}.
\end{itemize}
\end{remark}

\section{Linearization principle}

\subsection{Fractional semilinear Cauchy problem}\label{sec:gen:FRDS}

We prove the linearization principle for the following abstract 
fractional semilinear  equation with $\alpha\in (0,1]$
    \begin{equation}   \label{eq:FRDS}
                \partial^\alpha_t u =  \A  u + f(t,u), \qquad t>0,
    \end{equation}
 supplemented with an initial condition 
    \begin{equation}   \label{ini:FRDS}
                u(\cdot,0) = u_0,
    \end{equation}
 with a linear (possibly unbounded) operator $\A $.

\begin{theorem}\label{thm:exist}
    Let $\alpha\in (0,1]$. 
Assume that $\A $ generates a strongly continuous semigroup and 
$f: [0,\infty)\times \X \mapsto \X$ is a continuous function which is also locally Lipschitz continuous in the variable $u$.
 For every $u_0 \in \X$ there exists $T>0$ such that problem \eqref{eq:FRDS}-\eqref{ini:FRDS} has a unique mild solution $u \in C\big( [0,T];  \X \big)$.
    Moreover, if $T_{max}>0$ is a maximal time of the existence of this solution and 
    \begin{equation} \label{u:blowup}
   \text{if} \quad  T_{max}<\infty, \quad \text{then} \quad \sup_{t\in [0, T_{max})}\|u(t)\|=\infty.\end{equation}
\end{theorem}

\begin{proof}
    A mild local-in-time solution is constructed in the usual way via the Banach contraction principle applied to the Volterra integral equation 
    \begin{equation}\label{duhamel0}
        u(t) = S_\alpha (t) u_0 + \int_0^t (t-s)^{\alpha-1}P_\alpha (t-s) f( s, u(s))\, ds
    \end{equation} 
by using the resolvent estimates from Lemma~\ref{lem:SPE}.
\end{proof}

\begin{remark} \label{rem:exist} 
The monograph by Gal and Varma \cite[Sec.~3]{gal2020fractional} provides well-posedness results for the abstract Cauchy problem 
\eqref{eq:FRDS}-\eqref{ini:FRDS} developed in the same spirit as Rothe \cite{R84} who considered the classical parabolic
problem ($\alpha=1$) for second order elliptic operators in divergence form.
Analogous results on the existence of solutions to abstract fractional equations similar to the one in \eqref{eq:FRDS} can be found, for example, in
\cite{AL08,WCX12,ABCCN15,C-N22} and in the references therein.
\end{remark}


\subsection{Nonlinear stability}
 {\it The linearization principle} considered in this work states that the stability (or instability) of solutions to the linear equation $\partial^\alpha_tu=\A u$ implies (via the Taylor expansion) the stability (or instability) of solutions to the quasi-linear equation  \eqref{eq:FRDS}. Here, we recall the classical definition.

\begin{definition}\label{def:stab}
    The zero solution of equation \eqref{eq:FRDS} is stable if for every $\varepsilon>0$ there is $\delta>0$ such that for each initial datum $u_0\in\X$ with $\|u_0\|\le \delta$ the corresponding solution $u=u(t)$ of problem \eqref{eq:FRDS}-\eqref{ini:FRDS} is global in time and satisfies $\|u(t)\|\le \varepsilon$ for all $t>0$.
\end{definition}

In the following, we assume the linear stability by imposing decay estimates (discussed in Section \ref{sec:lin:stab}) on the resolvent families.
\begin{assumption}[Linear stability]\label{ass:stab:resolvent}
There exists a constant $K>0$ such that 
    the resolvent families  \eqref{Salpha} and \eqref{Palpha} satisfy
            \begin{equation} \label{SP:ass:stab}
               \|S_\alpha (t)v\|\le K \|v\| \min\{t^{-\alpha},1\} 
               \quad\text{and}\quad
               \|t^{\alpha-1}P_\alpha (t)v\|\le  K \|v\| \min\{t^{-\alpha-1},t^{\alpha-1}\},
           \end{equation}
        for  all $v\in\X$ and all $t> 0$. 
\end{assumption}

    \begin{theorem}[Nonlinear stability] \label{thm:stab}
        Let the decay estimates of the resolvent families from Assumption \ref{ass:stab:resolvent} hold true.
        Suppose that 
            \begin{equation}\label{f:flat}
                \lim_{\|u\|\to 0}\frac{\sup_{t\ge 0} \|f(t,u)\|}{\|u\|}=0.
            \end{equation}
         Then, the stationary solution $u \equiv 0$ of equation  \eqref{eq:FRDS}
         is stable. Moreover, there exist numbers $\delta>0$ and $C>0$ such that for all $u_0\in \X$ with $ \|u_0\|\leq \delta$, the corresponding solution of Cauchy problem \eqref{eq:FRDS}-\eqref{ini:FRDS} satisfies
         \begin{equation}
             \|u(t)\|\leq C \|u_0\| \min\{t^{-\alpha},1\} \quad \text{for all} \quad t> 0.
         \end{equation}
    \end{theorem}

    \begin{proof}
    By assumption \eqref{f:flat}, for every $\kappa>0$ there exists $R>0$ such that 
    \begin{equation}
        \text{if} \quad \|u\|<R\quad\text{then}\quad \|f(t,u)\|\leq \kappa\|u\|.
    \end{equation}
    We choose the parameters $R$ and $\kappa$ at the end of this proof.
    For $\delta \in (0,R)$ and for 
     $u_0\in \X$ satisfying  $\|u_0\|=\delta$, we denote by $u=u(t)$ the corresponding local-in-time solution of problem \eqref{eq:FRDS}-\eqref{ini:FRDS} provided by Theorem \ref{thm:exist}. In fact, this  solution exists for all $t\ge 0$, which is an immediate consequence 
     of the estimates below combined with the usual continuation argument involving 
      relation \eqref{u:blowup}.
        Define
            \begin{align}
                & T=\sup\left\{r>0\,:\, 
                \| u(t)-S_\alpha(t)u_0\|< \delta \min\{t^{-\alpha},1\}\quad \text{for all} \quad t\in [0,r] \right\}.
            \end{align} 
If $T=\infty$, the proof is complete because, by the definition of $T$ and by Assumption \ref{ass:stab:resolvent}, we have
\begin{equation}\label{in:stab:10}
\begin{split}
    \|u(t)\|&\le \| u(t)-S_\alpha(t)u_0\| +\|S_\alpha(t)u_0\|\\
    &\le \delta \min\{t^{-\alpha},1\}(1+K)\qquad \text{for all} \quad t\ge 0.
    \end{split}
\end{equation}

Now, we conjecture that $T<\infty$ in order to obtain a contradiction.
For all $\delta< R/(1+K)$, inequality \eqref{in:stab:10} implies that $\|u(\tau)\|<R$ for all $\tau \in [0,T]$.
           Thus,  
\begin{equation}\label{f:stab}
    \|f(\tau, u(\tau))\|\le \kappa\|u(\tau)\|\le \kappa 
    \big(
    \|u(\tau)-S_\alpha(\tau)  u_0\| +\|S_\alpha(\tau)  u_0\|
    \big).
\end{equation}
        We apply the norm $\|\cdot\|$ to integral equation \eqref{duhamel0} with $t\in [0,T]$.
By using estimates 
\eqref{SP:ass:stab},
inequality \eqref{f:stab},
the definition of $T$, and Lemma \ref{lem:integral} with $\eta=0$ below,
 we obtain
        \begin{equation}\label{stab:est}
            \begin{split}
                 \|u(t)-S_\alpha (t)u_0\| &\leq \kappa K \int_0^t  
                \min\big\{(t-s)^{-\alpha-1},(t-s)^{\alpha-1}\big\}\\
                &\hspace{3cm}\times \big(  \|u(s)-S_\alpha (s)u_0\| +\|S_\alpha (s)u_0\|
                \big)
                \, ds\\
                &\le \kappa K(1+K) C(\alpha) \delta \min\{t^{-\alpha}, 1\},
            \end{split}
        \end{equation}
with a constant $C(\alpha)=C(\alpha,0)>0$ from Lemma \ref{lem:integral}.
In particular,  by the definition of $T$ (under the conjecture that $T<\infty$) and by inequality \eqref{stab:est} with $t=T$, 
we obtain the relations 
$$
    \delta \min\{T^{-\alpha}, 1\} 
    =\| u(T)-S_\alpha(T)u_0\|
    \leq \kappa K(1+K)  C(\alpha) \delta \min\{T^{-\alpha}, 1\}
$$
which leads to contradiction for sufficiently small $\kappa>0$, namely, for 
$$\kappa< \frac{1}{K(1+K)  C(\alpha)}.$$
    Thus, $T=\infty$ and inequality \eqref{in:stab:10} holds true for all $t\geq 0$, where $\delta=\|u_0\|$.
    \end{proof}

The estimate of the following lemma is used not only in the proof of Theorem~\ref{thm:stab} but also in the proofs of other results of this work.

\begin{lemma}\label{lem:integral}
   Let $\alpha\in (0,1)$ and $\eta\ge 0$. There exists a number $C(\alpha, \eta)>0$ such that  
\begin{equation}\label{ineq:alpha:eta}
 \int_0^t  
        \min\{(t-s)^{-\alpha-1},(t-s)^{\alpha-1}\}\big(\min\{s^{-\alpha},1\}\big)^{1+\eta}\, ds\leq C(\alpha,\eta)
        \min \{t^{-\alpha},1\}
\end{equation}
        for all $t\ge 0$.
\end{lemma}
\begin{proof}
For $t\in (0,1]$, the integral in inequality \eqref{ineq:alpha:eta} is bounded by 
$$
\int_0^t(t-s)^{\alpha-1}\, ds = \frac{t^\alpha}{\alpha}\le \frac{1}{\alpha}.
$$

For $t>1$, we decompose the integral in \eqref{ineq:alpha:eta} as 
$\int_0^{t/2}...\, ds +\int_{t/2}^t...\, ds$ and,
by using the inequality $\min\{a,b\}\le a$,
we estimate first term by the quantity
\begin{align*}
     \int_0^{t/2}  
            (t-s)^{-\alpha-1}\cdot 1\, ds
            \le \left(\frac{t}2\right)^{-\alpha-1}\cdot \frac{t}2 = t^{-\alpha} 2^{\alpha}.
\end{align*}
On the other hand, the integral with respect to  $s\in [t/2,t]$ is estimated by 
 \begin{align*}
 \left(\frac{t}{2}\right)^{-\alpha(1+\eta)}
     &\int_{t/2}^t  
        \min\{ (t-s)^{-\alpha-1}, (t-s)^{\alpha-1}\}\, ds\\
        &= 
        \left(\frac{t}{2}\right)^{-\alpha(1+\eta)} \int_{0}^{t/2}  
        \min\{\tau^{-\alpha-1}, \tau^{\alpha-1} \}
         \, d\tau \le  
         \left(\frac{t}{2}\right)^{-\alpha(1+\eta)}C(\alpha),
 \end{align*}
where $C(\alpha)= \int_{0}^{\infty}  
        \min\{\tau^{-\alpha-1}, \tau^{\alpha-1} \}
         \, d\tau<\infty.$
\end{proof}

    \subsection{Nonlinear instability} \label{sec:instability}
Now, we are in a position to formulate conditions that lead to the instability of solutions.  
Recall that, in view of Definition \ref{def:stab}, to show the instability of the zero solution of equation \eqref{eq:FRDS}, it suffices to find $R_0>0$, a sequence of initial conditions $\{u_{0,k}\}_{k=0}^\infty$ satisfying $\|u_{0,k}\|\to 0$, and a sequence of times $T_k>0$ such that the corresponding solutions $u_k=u_k(t)$ satisfy $\|u_k(T_k)\|=R_0$ for all $k\in \mathbb{N}$.

We prove the instability of the zero solution of nonlinear equation \eqref{eq:FRDS} under
the following instability assumption of the zero solution of the linear equation $\partial_t^\alpha u=\A u$.

\begin{assumption}[Linear instability]\label{ass:instab}
The operator $(\A , D(\A ))$
has an eigenvalue $\lambda \in \C $ with the following properties
    \begin{itemize}
        \item 
    $\lambda\neq 0$ and  $\left|\arg(\lambda )\right| <\frac{\alpha\pi}{2}$;
     \item
     for every $\omega>0$ there exists $C=C(\omega)>0$
    such that 
    \begin{equation}\label{Sa:instab}
        \|S_\alpha(t)u_0\|
        \le C(\omega) e^{(\Re \lambda  +\omega)^{1/\alpha}t}\|u_0\|
    \end{equation}
    \text{and}
    \begin{equation}\label{Pa:instab}
        \|t^{\alpha-1} P_\alpha(t)u_0\|
        \le C(\omega) e^{(\Re \lambda  +\omega)^{1/\alpha}t}\|u_0\|,
    \end{equation}
    for all $t\ge 0$ and $u_0\in\X$.
    \end{itemize}
\end{assumption}

\begin{remark}
    Note that we do not require $\lambda$ in Assumption \ref{ass:instab} to satisfy the relation $\Re\lambda=s(\A)$, as discussed in Proposition \ref{prop:instab}; see Remark \ref{rem:instab} for more comments.
\end{remark}

\begin{remark}
For simplicity of presentation, we assume that $\lambda$ in Assumption \ref{ass:instab} is an eigenvalue.
In fact, following the reasoning of \cite{SS98} (see also \cite[Theorem 2.4]{CKKW21}), our instability result can be extended to the case where this number belongs to the boundary of the spectrum $\sigma(\A )$
and is an approximate eigenvalue.
\end{remark}

\begin{theorem}[Nonlinear instability]\label{thm:instab}
   Under  Assumption~\ref{ass:instab} and for sufficiently flat nonlinearities at the origin, namely, satisfying,
        \begin{equation}\label{f:flat:instab}
               \|f(t, \xi)\| \leq \kappa \|\xi\|^{1+\eta}  \quad \text{for all} \quad \|\xi\|\leq R \quad\text{and}\quad t\ge 0,
        \end{equation}
for some constants $\eta>0$, $\kappa>0$ and  $R>0$,
   the zero solution of  equation \eqref{eq:FRDS} is unstable.
\end{theorem}

\begin{proof}
    Fix $u_0 \in \X$ with $\|u_0\|= 1$ as the eigenfunction corresponding to the eigenvalue $\lambda$ from Assumption \ref{ass:instab}.
    We may assume that,
     for all sufficiently small $\delta \in (0,\min\{1,R\})$ (where $R$ is from assumption \eqref{f:flat:instab}),
    there exists 
    a unique global-in-time  solution $u_\delta \in C\big([0,\infty); \X\big) $ of equation \eqref{duhamel0} 
	with the initial datum $u_0$ replaced by $\delta u_0$. Moreover, we may also assume that this solution satisfies $\|u_\delta(t)\|< R$
for all $t\geq 0$; indeed,  the existence of a sequence $\delta_k\to 0$ and $T_k>0$ such that the corresponding solution satisfies 
$\|u_{\delta_k}(T_k)\|=R$ immediately implies the instability of zero solution (see Definition \ref{def:stab}).

Recall that $S_\alpha(\tau) u_0=E_\alpha(\lambda  \tau^\alpha)u_0$ (see comments on equation \eqref{Sa:eigen}).
We define two numbers
	\begin{equation}
		T 
		= 
		\sup 
		\left\lbrace 
		t>0
		: 
		\, 
		\| 
		u_\delta(\tau) - S_\alpha(\tau) \delta u_0 
		\|
		\le 
		\frac{\delta}{2}
        \left|E_\alpha(\lambda  \tau^\alpha)\right|
		\text{ for all } 
		\tau\in[0,t] 
		\right\rbrace 
		\label{eq:TDef}
	\end{equation}
    and $T_0>0$  such that $\delta \left|E_\alpha(\lambda  T_0^\alpha)\right|=2$. Such a number  $T_0>0$ exists for all small $\delta>0$ by 
    the exponential growth of $ \left|E_\alpha(\lambda  t^\alpha)\right|$ obtained 
   in Lemma \ref{lem:E:E} for $|\arg \lambda |<\alpha\pi/2$.

	If either $T>T_0$ or $T= \infty$, 
	then the zero solution is unstable.
	Indeed, 
	by the definitions of $T$ and $T_0$ combined with relation \eqref{Sa:eigen}, we obtain the inequality
	\begin{equation}
    \begin{split}
		\|
		u_\delta(T_0)
		\|
		&\ge 
		\|
		S_\alpha(T_0)
		\delta 
		 u_0
		\|
		- 
		\frac{\delta}{2} 
        \left|E_\alpha(\lambda  T_0^\alpha)\right|
        =\frac{\delta}{2} 
        \left|E_\alpha(\lambda  T_0^\alpha)\right| =1
        \end{split}
	\end{equation}
	which imply the instability of the zero solution.
	
	Next, we suppose that $T \le T_0$ and, for every $t\in [0,T]$, we consider the mild representation of equation \eqref{eq:FRDS} with the initial condition $\delta u_0$
	\begin{equation}\label{duh:instab}
		u_\delta(t) 
		- 
		S_\alpha(t)
		\delta 
		u_0 
		= 
		\int_0^t (t-\tau)^{\alpha-1} P_\alpha(t-\tau) f(\tau, u_\delta(\tau)) \, d\tau.   
	\end{equation}
By assumption \eqref{f:flat:instab}, there exist $\kappa>0$ and $C(\eta)>0$ such that 
for $\|u_\delta(\tau)\|<R$, we have
\begin{equation}\label{f:instab:2}
\begin{split}
    \|f(\tau, u_\delta(\tau))\|&\le \kappa\|u_\delta(\tau)\|^{1+\eta}\\ &\le \kappa C(\eta)
    \left(
    \|u_\delta(\tau)-S_\alpha(\tau) \delta u_0\|^{1+\eta} +\|S_\alpha(\tau) \delta u_0\|^{1+\eta}
    \right).
    \end{split}
\end{equation}
We apply the norm $\|\cdot\|$ to both sides of equation \eqref{duh:instab} for $t\in [0,T]$.
By using estimates 
\eqref{Sa:instab}-\eqref{Pa:instab} with arbitrary $\omega>0$,
inequality
\eqref{f:instab:2}, the definition of $T$ in \eqref{eq:TDef},
and relation \eqref{Sa:eigen}, we obtain
	\begin{equation}
        \begin{split}
		 \| 
		u_\delta(t)  &-  S_\alpha(t)\delta u_0 
		\| 
		 \\
		\le &
		C 
		\int_0^t 
        e^{(\Re\lambda +\omega)^{1/\alpha}(t-\tau)}
		\left(
		\| 
		u_\delta(\tau)  
        -
        S_\alpha(\tau)
		\delta 
		u_0 
		\|^{1+\eta} 
		+ 
		\| 
		S_\alpha(\tau)
		\delta 
		u_0
		\|^{1+\eta}
		\right) \, d\tau\\
		\le &
		C 
		\int_0^t 
		e^{(\Re\lambda +\omega)^{1/\alpha}(t-\tau)}
		\left(
		\frac{\delta^{1+\eta}}{2^{1+\eta}}
        |E_\alpha(\lambda  \tau^\alpha)|^{1+\eta}
		+ 
		\delta^{1+\eta} e^{(\Re\lambda)^{1/\alpha}(1+\eta)\tau}
		\right) \, d\tau\\
        \le &
		C \delta^{1+\eta}
		\int_0^t 
		e^{(\Re\lambda +\omega)^{1/\alpha}(t-\tau)}
		 e^{(\Re\lambda )^{1/\alpha}(1+\eta)\tau}
		\ \, d\tau.
		\label{eq:DuhDifEs:0}
        \end{split}
	\end{equation}
Now, choosing $\omega> 0$ such that 
$$
(\Re \lambda+\omega)^{1/\alpha} < (\Re \lambda)^{1/\alpha}(1+\eta)
$$
we obtain from  estimate \eqref{eq:DuhDifEs:0}
\begin{equation}
        \begin{split}
		 \| 
		u_\delta(t)  &-  S_\alpha(t)\delta u_0 
		\| 
  \le 
		C\delta^{1+\eta} t e^{(\Re\lambda)^{1/\alpha}(1+\eta) t}.
		\label{eq:DuhDifEs}
        \end{split}
        ,\qquad{\color{blue} 0\le t\le T\le T_0}
	\end{equation}
In particular, 
by the definition of the number $T$ and by inequality \eqref{eq:DuhDifEs} for $t = T$, 
	we have the relations
	\begin{equation}\label{rel:instab}
		\frac\delta2 
		 |E_\alpha(\lambda  T^\alpha)|
		= 
		\| 
		u_\delta(T) 
		- 
		S_\alpha(T) \delta u_0 
		\|
		\le 
		C
		\delta^{\eta+1} T
		e^{(\Re\lambda)^{1/\alpha}(1+\eta)T}.
	\end{equation}
Note that, for sufficiently small $\delta>0$,
there exists 
	the number $T_*\in (0,T]$ such that  
	\begin{equation}\label{Ea:equality}
\frac\delta2 
		 |E_\alpha(\lambda  T_*^\alpha)|
		= 
		C
		\delta^{\eta+1} T_*
		e^{(\Re\lambda)^{1/\alpha}(1+\eta)T_*}
	\end{equation}
    which is an immediate consequence of the lower bound of $E_\alpha(\lambda T)$ in \eqref{Ea:lower}
    (Notice that $T_*$ is arbitrarily large when $\delta\to 0$).
Thus, for $\delta>0$  small enough such that $\delta/2>C\delta^{\eta+1}$ with a constant $C>0$ from inequality \eqref{rel:instab},
by inequality \eqref{eq:DuhDifEs} with $t=T_*$ 
and by equations \eqref{Sa:eigen} and \eqref{Ea:equality} we have 
\begin{equation} \label{uT*:instab}
		\left\| 
		u(T_*) 
		\right\|
		\ge
		\|
		S_\alpha(T_*)
		\delta  
		u_0
		\|
		-
		C\delta^{1+\eta} T_* e^{(\Re\lambda +\omega)^{1/\alpha}(1+\eta)T_*}
		= 
		\frac{\delta}{2} |E_\alpha(\lambda  T_*^\alpha)|.
	\end{equation}
    
Now, we come back to equation \eqref{Ea:equality}. By using the lower bound for $E_\alpha(\lambda T_*)$ from expression \eqref{Ea:lower} for sufficiently small $\delta$, we obtain the inequality
$$
\frac\delta2 
		 |E_\alpha(\lambda  T_*^\alpha)|
		= 
		C
		\left(\delta T_*^{1/(1+\eta)}
		e^{(\Re\lambda)^{1/\alpha}T_*}\right)^{\eta+1}
        \leq C(\alpha,\lambda, \omega, \eta) \left(\frac\delta2 
		 |E_\alpha(\lambda  T_*^\alpha)|\right)^{\eta+2}
$$
which implies that 
$$\frac{\delta}{2}|E_\alpha(\lambda  T_*^\alpha)|\ge C>0 \quad \text{with}\quad  C= C(\alpha,\lambda, \omega, \eta)^{-1/(\eta+1)}>0
$$
independent of $\delta$ and of $T_*$. Applying this estimate on the right-hand side of inequality \eqref{uT*:instab} we complete the proof of instability of the zero solution of equation \eqref{eq:FRDS}.
\end{proof}


\section{Applications to fractional reaction-diffusion systems} \label{sec: Application}

\subsection{Fractional systems of differential equations}\label{sec:FDE}
We illustrate the abstract results from the previous sections by applying them to study stability of solutions of the initial-boundary value problem for systems of fractional reaction-diffusion equations. We begin with the simplest case of the following system of fractional differential equations
    \begin{align}  \label{eq: FDEsystem}
        \partial^\alpha_t v = g(v)
        \qquad\text{with }\ \alpha\in(0,1),
    \end{align}
where
    \begin{align}
        v=
            \begin{pmatrix}
                v_1(t,x)\\
                \vdots\\
                v_n(t,x)
            \end{pmatrix}
        \qquad\text{and}\qquad
        g(v)=
            \begin{pmatrix}
                g_1(v_1,...,v_n)\\
                \vdots\\
                g_n(v_1,...,v_n)
            \end{pmatrix} ,
    \end{align}
with $C^1$-nonlinearities $g_i$ for  $i\in\{1,\ldots,n\}$.
We assume that system \eqref{eq: FDEsystem} has a stationary solution, namely,
    \begin{equation}
        \bar{v}\in \mathbb{R}^n 
        \quad\text{such that}\quad
        g(\bar{v})=0.
    \end{equation}
Introducing the variables $u(t) = v(t)-\bar{v}$ and using the Taylor expansion, we obtain the fractional system
    \begin{equation}
        \partial^\alpha_t u = Au+ f(u),
        \label{eq: linearised}
    \end{equation}
where
    \begin{equation}
        A= 
        \begin{pmatrix}
            \frac{\partial g_1}{\partial v_1}(\bar{v}) &\ldots &\frac{\partial g_1}{\partial v_n}(\bar{v})\\
            \vdots &\ddots &\vdots\\
            \frac{\partial g_n}{\partial v_1}(\bar{v}) &\ldots &\frac{\partial g_n}{\partial v_n}(\bar{v})
        \end{pmatrix}
    \label{eq: linmatrix}
    \end{equation}
and $f$ is the corresponding Taylor reminder.

We are in a position to formulate the linearization principle for general systems of fractional differential equations 
\eqref{eq: FDEsystem}.
    
    \begin{theorem}[Nonlinear stability]\label{thm:stab:fde}
    Consider a stationary solution $\bar{v}$ of system \eqref{eq: FDEsystem}.
        \begin{enumerate}
            \item Assume that all eigenvalues $\lambda$ of the matrix $A$ in \eqref{eq: linmatrix} satisfy
                \begin{align*}
                    |\arg(\lambda)|>\frac{\alpha \pi}{2}.
                \end{align*}
            Then the stationary solution $\bar{v}$ is stable. Moreover, there exist constants $\delta>0$ and $C>0$ such that if $\|v(0)-\bar{v}\|\le \delta$, then the solution $v(t)$ of equation \eqref{eq: FDEsystem} with the initial datum $v(0)$ exists for all $t>0$ and satisfies
                \begin{align}
                    \|v(t)-\bar{v}\|\le C\|v(0)-\bar{v}\|\ \min\{t^{-\alpha},1\}.
                \end{align}

            \item Suppose that there exists an eigenvalue $\lambda$ of the matrix $A$ in \eqref{eq: linmatrix} such that
                \begin{align}
                    \lambda\ne 0
                    \qquad\text{and}\qquad
                    |\arg(\lambda)|<\frac{\alpha \pi}{2}.
                \end{align}
                Suppose, moreover, that the nonlinearity $g=g(v)$ has $C^2$- regularity.
            Then, the equilibrium solution $\bar{v}$ of equation \eqref{eq: linearised} is unstable.
        \end{enumerate}  
    \end{theorem}

\begin{remark}
Theorem \ref{thm:stab:fde} has already been proven in other papers; however, we present its complete proof to illustrate our general stability results in action. More precisely, the stability and instability theorems for linear systems of fractional differential equations of the same fractional order were first obtained by Matignon \cite{M98} (see also \cite[Theorem~7.20]{diethelm2010analysis} for detailed calculations). The qualitative theory, including the asymptotic behavior of solutions to multi-order systems of linear fractional differential equations, can be found in the work by Diethelm {\it et al.}~\cite{DiethelmSiegmundTuan2017}. The nonlinear stability of constant steady states for systems of fractional differential equations \eqref{eq: FDEsystem} is proved in the works \cite{ZTYC15, CDST16}, and the corresponding results on nonlinear instability were recently published in \cite{CDST17}. We refer to recent work \cite{CTT20} for other related comments and references.
\end{remark}

In order to apply our general stability results, we should first prove the estimates required in Assumption \ref{ass:stab:resolvent} and  Assumption \ref{ass:instab} for
the resolvent families defined in \eqref{Salpha} and \eqref{Palpha},
where 
\begin{equation}\label{T:A}
    T(t)=e^{tA}=\sum_{k=0}^\infty \frac{t^kA^k}{k!}.
\end{equation}

\begin{theorem}[Resolvent estimates] \label{lem:Jordan}
Let $\alpha\in (0,1)$
and $A$ be an arbitrary $n\times n$ matrix with constant coefficients.
Denote by $\sigma(A)$ the set of eigenvalues of matrix $A$.
Consider the resolvent families $S_\alpha(t)$ and $P_\alpha(t)$ given by formulas \eqref{Salpha} and \eqref{Palpha}, where $T(t)$ is the semigroup \eqref{T:A}.
    \begin{enumerate}
        \item  If $\sigma(A)\subset \{\lambda\in\C\,:\, |\arg(\lambda)|>\frac{\alpha \pi}{2} \}$,
         then there exists a constant $K>0$ such that 
            \begin{equation} \label{Sa:stab:fde}
               \|S_\alpha (t)u_0\|\le K \|u_0\| \min\{t^{-\alpha},1\} 
            \end{equation}   
               \text{and}
            \begin{equation} \label{Pa:stab:fde}
               \|t^{\alpha-1}P_\alpha (t)u_0\|\le  K \|u_0\| \min\{t^{-\alpha-1},t^{\alpha-1}\},
           \end{equation}
        for  all $u_0\in\R^n$ and all $t> 0$. 

         \item  Assume $\{\lambda\in\sigma(A)\,:\, |\arg(\lambda)|<\frac{\alpha \pi}{2} \}\neq\emptyset$.
         Let $\lambda \in \sigma(A)$ satisfy 
         $$
         \Re \lambda = \max \left\{ \Re \mu \;:\: \mu \in \sigma(A), \quad |\arg(\mu)|<\frac{\alpha \pi}{2}\right\}.
         $$
         Then, for every $\omega>0$ there exists $C=C(\omega)>0$
    such that 
        \begin{equation}\label{Sa:instab:fde}
            \|S_\alpha(t)u_0\|
            \le C(\omega) e^{(\Re \lambda  +\omega)^{1/\alpha}t}\|u_0\|
        \end{equation}
    \text{and}
        \begin{equation}\label{Pa:instab:fde}
            \|t^{\alpha-1}P_\alpha(t)u_0\|
            \le C(\omega) e^{(\Re \lambda  +\omega)^{1/\alpha}t}\|u_0\|,
        \end{equation}
    for all $t\ge 0$ and $u_0\in\R^n$.
    \end{enumerate}
\end{theorem}

\begin{proof}
The idea used in this proof appeared already in \cite[Theorems 7.13 and 7.14]{diethelm2010analysis} and in \cite[Theorem 3.1]{DiethelmSiegmundTuan2017}. Here, we recall this reasoning in order to state the results in a form suitable for our applications.

\textit{Step 1. The matrix $A$ is a Jordan block.}
Assume first that
\begin{equation} \label{A:lambda}
A=    
 \begin{pmatrix}
                        \lambda       &1            &\ldots  &0      \\
                        0       &\lambda           &\ldots  &0      \\
                        \vdots  &\vdots  &\ddots  &\vdots \\
                        0       &0           &\ldots  &\lambda      
                    \end{pmatrix}.
\end{equation}
with some $\lambda\in \C$.
For every $u_0=\big(u^0_1,\cdots,u^0_n\big)\in\R^n$ and for matrix \eqref{A:lambda},
the resolvent 
$$u(t)=\big(u_1(t), ..., u_n(t)\big)=S_\alpha (t)u_0$$ satisfies the initial value problem \eqref{Sa:Cauchy}
which, in the case of the Jordan block \eqref{A:lambda}, reduces to the Cauchy problem for the system of fractional differential equations 
            \begin{align*}
                \partial_t^\alpha u_i &= \lambda_i u_i + u_{i+1}, \quad i\in \{1,\ldots, n-1\},\\ 
                \partial_t^\alpha u_n &= \lambda_n u_n.
            \end{align*}
Consequently, by the Duhamel principle \eqref{Duh:0},
    \begin{align}
       u_i(t) &= E_\alpha(\lambda t^\alpha) u^0_i + \int_0^t (t-s)^{\alpha-1} E_{\alpha,\alpha}\big(\lambda (t-s)^\alpha\big) u_{i+1}(s) \, d s,\quad i\in \{1,\ldots, n-1\},\\ 
       u_n(t) &= E_\alpha(\lambda t^\alpha) u^0_n.
    \end{align}

 If $\lambda\in\C$ satisfies $|\arg(\lambda)|>\frac{\alpha \pi}{2}$, then by Lemma \ref{lem:ML} part (1),
 \begin{align}
     \left\|E_\alpha(\lambda t^\alpha) u_0^i\right\|&\le m(\alpha, \lambda) \min\{t^{-\alpha},1\}\|u_0^i\|, \\
      \left\|t^{\alpha-1}E_{\alpha,\alpha}(\lambda t^\alpha)u_{i+1}\right\|&\le m(\alpha, \lambda) \min\{t^{-\alpha-1},t^{\alpha-1}\}\|u_{i+1}\|.
 \end{align}
 Thus, by induction from $i=n-1$ to $i=1$, we obtain
 \begin{align}
     \|u_i(t)\| &= \|E_\alpha(\lambda t^\alpha) u_0^i\| + \int_0^t \big\|(t-s)^{\alpha-1} E_{\alpha,\alpha}\big(\lambda (t-s)^\alpha\big) u_{i+1}(s)\big\| \, d s\\
     &\le m(\alpha, \lambda) \|u_0\|\left(  \min\{t^{-\alpha},1\} +\int_0^t  \min\{(t-s)^{-\alpha-1},(t-s)^{\alpha-1}\}
     \min\{s^{-\alpha},1\}\, ds \right)\\
     &\le C(\alpha,\lambda) \|u_0\|\min\{t^{-\alpha},1\},
 \end{align}
where the last inequality holds true by Lemma \ref{lem:integral}.

Estimates the internal resolvent $P_\alpha(t)$ are analogous, but now we use the Cauchy problem \eqref{Pa:Cauchy} for
$t^{\alpha-1}P_\alpha(t)$,
involving the fractional Riemann-Liouville derivative \eqref{RL:derivative}.
Thus, 
for every $v_0=\big(v_0^1,\cdots,v_0^n\big)\in\R^n$ and for matrix \eqref{A:lambda}, the vector 
 $$v(t)=\big(v_1(t), ..., v_n(t)\big)=t^{\alpha-1}P_\alpha (t)v_0,$$ satisfies the system 
            \begin{equation}\label{RL:vi:1}
            \begin{split}
                \RL\partial_t^\alpha v_i &= \lambda_i v_i + v_{i+1}, \quad i\in \{1,\ldots, n-1\},\\ 
                \RL\partial_t^\alpha v_n &= \lambda_n v_n.
                \end{split}
            \end{equation}
supplemented with the Cauchy type initial condition (involving the Riemann-Liouville fractional integral \eqref{fractional integral})
    \begin{equation}\label{RL:vi:2}
        \lim_{t\to 0} J^{1-\alpha}_t v_i(t)=v^i_0 ,\quad i\in \{1,\ldots, n\}.
    \end{equation}
Now, we recall (see, for example, \cite[Lemma 5.2]{diethelm2010analysis} or \cite[Ch.~7.2.1]{GKMR20}) that for  $\alpha \in (0,1)$ the 
Cauchy problem 
    \begin{equation}
        \RL\partial_t^\alpha v=\lambda v+f(t), 
        \quad  
        \lim_{t\to 0} J_t^{1-\alpha} v(t)  =v_0,
    \end{equation}
has the solution of the following form 
    \begin{equation}\label{duh:RL}
          v(t) = t^{\alpha-1}E_{\alpha,\alpha}(\lambda t^\alpha) v_0 + \int_0^t (t-s)^{\alpha-1} E_{\alpha,\alpha}\big(\lambda (t-s)^\alpha\big) f(s) \, d s.
    \end{equation}
Applying formula \eqref{duh:RL} to problem \eqref{RL:vi:1}-\eqref{RL:vi:2} we obtain the estimates of $t^{\alpha-1}P_\alpha(t)$
by proceeding analogously as in the first part of this proof.

 The proof is shorter for $\lambda\in\C$ in the matrix \eqref{A:lambda} satisfying $\lambda\neq 0$ and $|\arg(\lambda)|\le \frac{\alpha \pi}{2}$.
  In this case, we have $\Re\lambda>0$ (because $\alpha\in (0,1)$) and it is well known that (for the matrix $A$ in Jordan form \eqref{A:lambda}) for each $\omega>0$ there exists $C=C(\omega, \lambda)$ such that 
  \begin{equation}
     \|T(t)u_0\|=\|e^{tA}u_0\|\le C(\omega, \lambda) e^{(\Re \lambda+\omega) t}.
  \end{equation}
  Thus, the estimates of $S_\alpha(t)$ and $P_\alpha(t)$ are obtained from Lemma \ref{lem:SPE} combined with
the subordination formulas for the Mittag-Leffler functions in Lemma \ref{lem:E:E} and with
the inequalities in Lemma \ref{lem:ML}, Part (2).

\textit{Step 2. Arbitrary matrix $A$.}
Since, there exits an invertible matrix $P$ such that $A=PJP^{-1}$, where the matrix $J$ consists of Jordan blocks \eqref{A:lambda}, by using the series $T(t)$ from \eqref{T:A} in the formulas for $S_\alpha(t)$ and $P_\alpha(t)$ in \eqref{Salpha} and \eqref{Palpha}, our proof reduces to the analysis of Jordan blocks only.

If all eigenvalues of the matrix $A$ satisfy $|\arg(\lambda)|>\frac{\alpha \pi}{2}$, then by the stability estimates of each Jordan block \eqref{A:lambda} obtained in Step 1, complete the proof of inequalities \eqref{Sa:stab:fde} and \eqref{Pa:stab:fde}.

In the case of the instability estimates, 
we divide all eigenvalues of $A$ into two sets
\begin{equation}
    \sigma_{st}(A) =\left\{\lambda\in \sigma(A) \,:\, |\arg(\lambda)|\ge \frac{\alpha \pi}{2}\right\}, \quad  \sigma_{un}(A) =\left\{\lambda\in \sigma(A) \,:\, |\arg(\lambda)|< \frac{\alpha \pi}{2}\right\}.
\end{equation}
By assumption, among the eigenvalues in $\sigma_{un}(A)$, there is one with the maximal strictly positive real part, which we denote by $\lambda$ and which satisfies $\lambda\ne0$ and $|\arg(\lambda)|<\frac{\alpha \pi}{2}$.
The resolvents $S_\alpha(t)$ and $P_\alpha(t)$ restricted to subspaces corresponding to Jordan blocks with eigenvalues from $ \sigma_{st}(A)$
are uniformly bounded in $t>0$ (and decay in time).
Restricting the resolvents to subspaces corresponding to the eigenvalues from $ \sigma_{un}(A)$
we obtain the exponential estimates \eqref{Sa:instab:fde} and \eqref{Pa:instab:fde}.
 \end{proof} 
    \begin{proof}[Proof of Theorem \ref{thm:stab:fde}.]
The stability of the constant solution $\bar{v}$ of equation \eqref{eq: FDEsystem} is obtained by applying Theorem \ref{thm:stab}
to the modified problem \eqref{eq: linearised}. Decay estimates of the resolvent families $S_\alpha(t)$ and $P_\alpha(t)$ required by Assumption \ref{ass:stab:resolvent} are obtained in Theorem \ref{lem:Jordan}.
The Taylor reminder $f=f(u)$ in system \eqref{eq: linearised} satisfies assumption~\eqref{f:flat}.

Analogously, the instability of $\bar{v}$ is directly derived from Theorem \ref{thm:instab}.
 Here, we require the $C^2$-regularity of the nonlinear term $g=g(v)$ to guarantee the Taylor reminder $f=f(u)$ to satisfy assumption \eqref{f:flat:instab} with $\eta=1$.
    \end{proof}

\subsection{Constant solutions to fractional reaction-diffusion systems}\label{sec:FRD}
Next, we  apply the linearization principle to the fractional-reaction diffusion system with $\alpha\in (0,1)$ 
\begin{equation}\label{eq:FRD}
\begin{split}
    \partial_t^\alpha  v_1 &=   D_1\Delta v_1 +g_1(v_1,\dots,v_n),     \qquad x\in\Omega, \quad t>0,\\
    \vdots&\qquad \vdots \qquad\qquad\vdots \\
    \partial_t^\alpha  v_n &=   D_n\Delta v_n +g_n(v_1,\dots,v_n),  \qquad    x\in\Omega, \quad t>0,
\end{split}
\end{equation}
considered on a bounded open domain $\Omega \subset\R^n$ 
with a $C^2$-boundary $\partial\Omega$,
supplemented either 
with the Dirichlet boundary condition
    $$
    v_i=0, \qquad x\in \partial\Omega, \quad t>0, \quad i \in \{1,\dots,n\},
    $$
    or the Neumann boundary condition
    $$
    \nu\cdot \nabla v_i=0, \qquad x\in \partial\Omega, \quad t>0, \quad i \in \{1,\dots,n\}
    $$
and with an initial datum. Constant diffusion coefficients satisfy $D_i>0$ for all $i\in \{1,\dots,n\}$.

Recall that the Laplace operator with the suitable boundary condition and with the domain 
\begin{equation}\label{Delta:domain}
    D(\Delta) =\{u\in \bigcap_{p\ge 1} W^{2,p}_{loc}(\Omega)\,:\, u, \Delta u \in C(\overline\Omega), \; u \text{ satisfies the boundary conditions}\}
\end{equation}
generates (see, for example, \cite[Corollary 3.1.24]{L95} and references therein) a strongly continuous semigroup on the Banach space
    \begin{equation}\label{X:C}
        \X=
            \begin{cases}
                C(\overline\Omega),& \text{for the case of the Neumann boundary condition,}\\
                C_0(\overline\Omega),& \text{for the case of the Dirichlet boundary conditions,} 
            \end{cases}
    \end{equation}
supplemented with the usual norm $\|u\|_\infty=\max_{x\in \overline\Omega}|u(x)|$.
Thus, by Theorem \ref{thm:exist}, the initial-boundary value problem for system \eqref{eq:FRD} 
has a unique local-in-time solution for every
initial datum $u(0)\in \X^n$.

Now, assume that the boundary value problem for system \eqref{eq:FRD} has a constant stationary solution
    \begin{equation}
        \bar{v}\in \mathbb{R}^n 
        \quad\text{such that}\quad
        g_i(\bar{v})=0 \quad \text{for} \quad i\in \{1,\dots,n\}.
    \end{equation}
Notice that in the case of the Dirichlet boundary condition, only $\bar v=0$ is allowed.
As the usual practice, the variable $u(t) = v(t)-\bar{v}$ satisfies the system 
    \begin{equation} \label{eq:FRD:constant}
        \begin{pmatrix}
            \partial_t^\alpha u_1\\
            \vdots\\
            \partial_t^\alpha u_n
        \end{pmatrix}=
        \begin{pmatrix}
            D_1\Delta u_1\\
            \vdots\\
            D_n \Delta u_n
        \end{pmatrix} +
        A
        \begin{pmatrix}
    	u_1\\
            \vdots\\
            u_n
        \end{pmatrix}
        +f  
        \begin{pmatrix}
    	u_1\\
            \vdots\\
            u_n
        \end{pmatrix},
    \end{equation}
where
    \begin{equation}
        A= 
        \begin{pmatrix}
            \frac{\partial g_1}{\partial v_1}(\bar{v}) &\ldots &\frac{\partial g_n}{\partial v_n}(\bar{v})\\
            \vdots &\ddots &\vdots\\
            \frac{\partial g_n}{\partial v_1}(\bar{v}) &\ldots &\frac{\partial g_n}{\partial v_n}(\bar{v})
        \end{pmatrix}
        \quad \text{and} \quad f  \begin{pmatrix}
	u_1\\
        \vdots\\
        u_n
	\end{pmatrix} \quad \text{is the Taylor reminder.}
    \label{eq:lin:constant}
    \end{equation}
We define the associated linear 
 operator
 \begin{equation} \label{eq:A:FRD}
 \A \begin{pmatrix}
	u_1\\
        \vdots\\
        u_n
	\end{pmatrix} =
     \begin{pmatrix}
	D_1\Delta u_1\\
        \vdots\\
        D_n \Delta u_n
    \end{pmatrix} +
    A
    \begin{pmatrix}
	u_1\\
        \vdots\\
        u_n
    \end{pmatrix}
    ,
\end{equation}
supplemented either with the zero Dirichlet boundary condition or with the Neumann boundary condition and 
with the domain 
\begin{equation}\label{L:domain}
    D(\A) =\{u\in \bigcap_{p\ge 1} W^{2,p}_{loc}(\Omega)\,:\, u, \A u \in C(\overline\Omega), \; u \text{ satisfies the boundary conditions}\}.
\end{equation}

We are in a position to formulate our main result on stability and instability of constant solutions to system \eqref{eq:FRD} studied in the Banach space $\X$ defined in \eqref{X:C}.

    \begin{theorem}\label{thm:stab:FRD}
    Consider a constant stationary solution $\bar{v}\in\R^n$ of the initial-boundary value problem for system \eqref{eq:FRD}.
        \begin{enumerate}
            \item Assume that all eigenvalues $\lambda$ of the operator $\big(\A, D(\A)\big)$ defined in \eqref{eq:A:FRD}-\eqref{L:domain} 
satisfy
                \begin{align*}
                    |\arg(\lambda)|>\frac{\alpha \pi}{2}.
                \end{align*}
            Then the stationary solution $\bar{v}$ is stable. Moreover, there exist constants $\delta>0$ and $C>0$ such that if $\|v(0)-\bar{v}\|_\infty\le \delta$, then the solution $v(t)$ of system \eqref{eq:FRD} with the initial datum $v(0)$ exists for all $t>0$ and satisfies
                \begin{align}
                    \|v(t)-\bar{v}\|_\infty\le C\|v(0)-\bar{v}\|_\infty\ \min\{t^{-\alpha},1\}.
                \end{align}

            \item Assume that
            $$
            \left\{\lambda \in \sigma(\A)\;:\;  \lambda\ne 0
                    \quad\text{and}\quad
                    |\arg(\lambda)|<\frac{\alpha \pi}{2}\right\}\neq\emptyset
            $$
                Suppose, moreover, that the nonlinearity $g=g(v)$ has  $C^2$- regularity.
            Then the equilibrium solution $\bar{v}$ of equation \eqref{eq:FRD} is unstable.
        \end{enumerate}  
    \end{theorem}

In the proof of Theorem \ref{thm:stab:FRD}, we use certain well-known properties of the eigenvalues of the operator $(\A, D(\A))$.
In the following lemma, the sequence $\{w^k\}_{k=0}^{\infty}$ denotes the orthonormal basis  of $L^2(\Omega)$ consisting of the eigenfunctions of the Laplace operator $-\Delta$, subject to the Dirichlet boundary condition or the Neumann boundary condition and we denote by $\{\mu^k\}_{k=0}^{\infty} \subset [0, \infty)$ the corresponding eigenvalues.

\begin{lemma} \label{lem:ADk}
Denote by $\sigma(\A)$ the spectrum of $(\A,D(\A))$ defined by  \eqref{eq:A:FRD}-\eqref{L:domain}.
\begin{enumerate}
    \item The spectrum $\sigma(\A)$ is discrete, and there exist constants $\omega \in \mathbb{R}$ and $\theta \in (\frac{\pi}{2}, \pi)$ such that
$$ \sigma(\A) \subset \{\lambda \in \mathbb{C} \,:\, |\arg(\lambda - \omega)| \ge \theta \}. $$

\item Define the matrices (with the eigenvalues $\mu^k$ of the operator $-\Delta)$
$$ A_{D,k} \equiv \begin{pmatrix} -D_1 \mu^k & \cdots & 0 \\ \vdots & & \vdots \\ 0 & \cdots & -D_n \mu^k \end{pmatrix} + A
\quad\text{for} \quad k\in\N\cup\{0\}.
$$
Then
\begin{equation}
    \sigma(\A) = \bigcup_{k=0}^\infty \sigma(A_{D,k}).
\end{equation}

\item There exist numbers $k_0 \in \mathbb{N}$ and $\omega_0 > 0$ such that all eigenvalues of the matrices $A_{D,k}$ with $k > k_0$ belong to the set
$\{\lambda \in \mathbb{C} \,:\, \Re \lambda \le -\omega_0 \}. $
\end{enumerate}
\end{lemma}

\begin{proof}
{\it Part 1.} These are well-known properties of elliptic operators and sectorial operators; see, for example, \cite[Corollary 3.1.21]{L95}.

{\it Part 2.} 
Let $\lambda\in \C$ be an eigenvalue of $\A$ with the eigenfunction
    \begin{equation}
        \bar z=\begin{pmatrix}
            z_1\\ \vdots \\ z_n
        \end{pmatrix} \in L^2(\Omega)^n,
    \end{equation}
and define $\bar v\in\R^n$ as
    \begin{equation}
        \bar v=
        \begin{pmatrix}
            v_1\\\vdots\\v_n
        \end{pmatrix}
        \equiv\begin{pmatrix}
            \int_\Omega z_1(x)w^k(x)\,dx\\ \vdots \\ \int_\Omega z_n(x)w^k(x)\,dx
        \end{pmatrix}
        ,
    \end{equation}
where $w^k$ is an eigenfunction of $-\Delta$ chosen in such a way that $\bar v\neq 0$.
The $i$-th equation in the system $(\A-\lambda I)\bar z=0$, for $\A$ given by \eqref{eq:A:FRD}, has the form 
    \begin{equation}
        D_i \Delta z_i +\sum_{j=1}^n a_{ij} z_j -\lambda z_i=0,
    \end{equation}
supplemented with the boundary condition.
Computing the $L^2(\Omega)$-scalar product of this equation with the eigenfunction $w^k$ and integrating by parts we obtain
    \begin{equation}
        -D_i\mu^k v_i+\sum_{j=1}^n a_{ij} v_j -\lambda v_i=0.
    \end{equation}
Hence, $\bar v$ is an eigenvector of the matrix $A_{D,k}$ with the eigenvalue $\lambda$.

On the other hand, let $\lambda \in \mathbb{C}$ be an eigenvalue of the matrix $A_{D,k}$ for some $k\in \N\cup\{0\}$ and denote by $\bar v\in \R^n$ the corresponding eigenvector. 
We are going to prove that $\lambda$ is an eigenvalue of $\A$ corresponding to the eigenvector $\bar v w^k$, where $w^k$ is the eigenfunction of $-\Delta$ with the eigenvalue $\mu^k$. Indeed, by a direct calculation, we obtain 
$$ (\A - \lambda I) (\bar v w^k)
=  w^k \left[\begin{pmatrix} -D_1 \mu^k & \cdots & 0 \\ \vdots & & \vdots \\ 0 & \cdots & -D_n \mu^k \end{pmatrix} 
 + A-\lambda I \right]\bar v =0. $$

{\it Part 3.}
Let $\lambda \in \mathbb{C}$ be an eigenvalue of $A_{D,k}$ for some $k \in \mathbb{N}$ and let $v \in \mathbb{R}^n$ be the corresponding eigenvector such that $\|v\| = 1$.
Then, using the explicit form of $A_{D,k}$, we obtain
$$ |\lambda| = \|\lambda v\| = \|A_{Dk} v\| \ge \mu^k \cdot \min\{D_1, \ldots, D_n\}-\|A\|. $$
The proof is complete by the fact that $\mu^k\to\infty $ as $k\to\infty$, because all the eigenvalues of $A_{D,k}$ form a discrete set with no accumulation points and because they stay in the sector described in Part 1. 
\end{proof}

\begin{proof}[Proof of Theorem \ref{thm:stab:FRD}.]
To prove that Assumptions \ref{ass:stab:resolvent} and \ref{ass:instab} are met,  
we will use the subordination formulas \eqref{Salpha} and \eqref{Palpha} for the resolvents 
 $S_\alpha(t)$ and $P_\alpha(t)$. Thus, we analyze the semigroup $T(t)$ of linear operators  generated by 
 the operator $\big(\A, D(\A)\big)$ on the Banach space $\X^n$.

Recall that the function
$u(t)=T(t)u_0$ with arbitrary $u_0\in \X^n$ satisfies the initial-boundary value problem 
    \begin{equation} \label{eq:Salpha}
         \begin{pmatrix}
    	\partial_t u_1\\
            \vdots\\
            \partial_t u_n
    	\end{pmatrix}=
         \begin{pmatrix}
    	D_1\Delta u_1\\
            \vdots\\
            D_n \Delta u_n
    	\end{pmatrix} +
        A
        \begin{pmatrix}
    	u_1\\
            \vdots\\
            u_n
        \end{pmatrix}, 
        \qquad 
        u(0)=u_0=
        \begin{pmatrix}
    	u_{0,1}\\
            \vdots\\
            u_{0,n}
        \end{pmatrix}
        ,
    \end{equation}
supplemented with the suitable boundary conditions. This problem has the explicit solution
    \begin{equation} \label{u:formula}
        u(t) =\sum_{k=0}^\infty \beta^k(t) w^k,
    \end{equation}
with the eigenfunctions $w^k$ of $-\Delta$ subject to the boundary condition and 
where the time dependent vector coefficients $\beta^k(t)$ satisfying the following systems of  differential equations 
    \begin{equation}\label{beta:system}
        \begin{pmatrix}
    	\partial_t \beta_1^k\\
            \vdots\\
            \partial_t \beta_n^k
    	\end{pmatrix}=
         \left[
         \begin{pmatrix} -D_1 \mu^k & \cdots & 0 \\ \vdots & & \vdots \\ 0 & \cdots & -D_n \mu^k \end{pmatrix}
          +
        A\right]
        \begin{pmatrix}
    	\beta_1^k\\
            \vdots\\
            \beta_n^k
    	\end{pmatrix}, \qquad k\in \{0,1,2,\dots\},
    \end{equation}
and the initial conditions
    \begin{equation}
    \beta^k(0)=
         \begin{pmatrix}
    	\beta_1^k(0)\\
            \vdots\\
            \beta_n^k(0)
    	\end{pmatrix}= 
         \begin{pmatrix}
    	\int_\Omega u_{0,1}w^k\, dx\\
            \vdots\\
            \int_\Omega u_{0,n}w^k\, dx
    	\end{pmatrix}.
    \end{equation}
Thus, 
$
    \beta^k(t) =e^{tA_{D,k}}\beta^k(0)
$ 
with the matrix $A_{D,k}$ from Lemma \ref{lem:ADk}. 

Now, for the numbers $k_0 \in \mathbb{N}$ and $\omega_0 > 0$ from Lemma \ref{lem:ADk}, Part 3, we decompose the sum in \eqref{u:formula}
\begin{equation}
    u(t)=T(t)u_0 =\sum_{k=0}^{k_0} e^{tA_{D,k}}\beta^k(0)w^k + \sum_{k=k_0+1}^\infty  e^{tA_{D,k}}\beta^k(0)w^k \equiv u_1(t)+u_2(t).
\end{equation}
Thus, using the subordination formula \eqref{Salpha} we obtain
\begin{align}\label{Salpha:sum}
    S_\alpha(t)u_0 &=\int_0^\infty \Psi_\alpha(s)T(st^\alpha)\;ds \\
    & = \sum_{k=0}^{k_0} \int_0^\infty \Psi_\alpha(s)e^{st^\alpha A_{D,k}}\beta^k(0)w^k ds +  \int_0^\infty \Psi_\alpha(s)u_2(st^\alpha)\;ds.
\end{align}

In both sums, the last term $u_2(t)$ corresponds to the operator $(\A,D(\A))$ restricted to the Banach space spanned by the eigenvectors 
$\{w^{k_0+1}, \dots, \}$, where, by Lemma \ref{lem:ADk}, Part 3, its spectral bound is bounded from above by $-\omega_0$.
Thus, for every $\varepsilon>0$,
\begin{equation}
    \|u_2(t)\|_\infty\le C(\omega_0,\varepsilon) e^{(-\omega_0+\varepsilon)t}\|u_0\|_\infty.
\end{equation}
Consequently, by Corollary \ref{cof:negative spectal bound},
\begin{equation}
   \left\| \int_0^\infty \Psi_\alpha(s)u_2(st^\alpha)\;ds\right\|_\infty\le C\|u_0\|_\infty \min\{t^{-\alpha},1\}.
\end{equation}

On the other hand, the first term on the right-hand side of equation \eqref{Salpha:sum}, containing the finite sum, is estimated directly by applying Theorem \ref{lem:Jordan}. In particular, to show the instability of the zero solution, we choose the eigenvalue $\lambda=\lambda_{max}$ 
  required in Assumption \ref{ass:instab} to be the one satisfying
$$
\Re\lambda_{max}=\max \left\{\Re \lambda\,:\, \lambda\in \sigma(\A) \quad \text{and} \quad |\arg(\lambda)|<\frac{\alpha \pi}{2}\right\},
$$
which exists by the assumption of Theorem \ref{thm:stab:FRD}, Part 2. 

Finally, applying Theorem \ref{thm:stab}, we complete the proof of Theorem \ref{thm:stab:FRD}, Part 1, and applying Theorem \ref{thm:instab}, we complete the proof of Theorem \ref{thm:stab:FRD}, Part 2.
\end{proof}



   \subsection{Turing instability}
We use Theorem \ref{thm:stab:FRD} to extend the concept of diffusion-driven instability to the fractional setting. This phenomenon, first described by Turing \cite{T52}, is the well-known mechanism explaining the emergence of stable patterns (see, for example, \cite[Chapter II]{Murray-2}).
In the following theorem, we show the Turing instability in the case of a system of two fractional reaction-diffusion equations
        \begin{align}\label{eq:Turing}
            &\partial^\alpha_t u = \su \Delta u + f(u,v), & x\in\Omega, \quad t>0,\\
            &\partial^\alpha_t v = \sv \Delta v + g(u,v), &  x\in\Omega, \quad t>0,
        \end{align}
    on a bounded domain $\Omega\subset\mathbb{R}^n$ 
    and supplemented either with the Dirichlet boundary condition
    \begin{equation}\label{Dirichlet}
    u=v=0, \qquad x\in \partial\Omega, \quad t>0,
    \end{equation}
    or the Neumann boundary condition
    \begin{equation}\label{Neumann}
    \nu\cdot \nabla u=  \nu\cdot \nabla  v=0, \qquad x\in \partial\Omega, \quad t>0.
    \end{equation}
Rather than studying Turing instability in its most general form, our goal in this work is just to 
 illustrate some applications of the general instability Theorem \ref{thm:instab}, formulated in this particular setting in Theorem \ref{thm:stab:FRD}.

    \begin{theorem}[Turing instability]\label{thm:Turing}
        Let $\alpha\in(0,1)$ and $f=f(u,v)$ and $g=g(u,v)$ be arbitrary $C^2$-nonlinearities.
        Assume that $(\bar{u},\bar{v})\in\mathbb{R}^2$ is an asymptotically stable, constant, stationary solution of the following fractional differential system
            \begin{align}
                \partial^\alpha_t u =  f(u,v),\\
                \partial^\alpha_t v =  g(u,v),
            \end{align} 
            in the sense that the matrix
            \begin{equation}\label{matrix:A:fg}
                A= \left(
                    \begin{array}{cc}
                         f_u(\bar{u},\bar{v})& f_v(\bar{u},\bar{v}) \\
                         g_u(\bar{u},\bar{v})& g_v(\bar{u},\bar{v})
                    \end{array}
                \right)
            \end{equation}
            have both  eigenvalues $\lambda_\pm\in\C$ satisfying 
                $
                    |\arg(\lambda)|>{\alpha \pi}/{2}
                $ {\rm(}see Theorem \ref{thm:stab:fde}, Part 1{\rm )}.
        If $f_u(\bar{u},\bar{v})>0$ and $\su>0$ are sufficiently small, then $(\bar{u},\bar{v})\in\mathbb{R}^2$ is an unstable solution of the initial-boundary value problem for the fractional reaction-diffusion system \eqref{eq:Turing}.
    \end{theorem}

\begin{remark} 
Theorems \ref{thm:stab:FRD} and \ref{thm:Turing} can be used to extend several results on Turing pattern formation obtained through linear stability analysis and numerical simulations in a reaction-diffusion system of several species where fractional  temporal derivatives operate on all species; see, for example, \cite{NN08,HLW05,GDMB09,GD10,DG18} and the references therein. By Theorems \ref{thm:stab:FRD} and \ref{thm:Turing}, constant stationary solutions of fractional reaction-diffusion systems considered in the cited papers are not only linearly unstable (as demonstrated in the cited papers by an analysis of the eigenvalues of the linearization operator) but are also unstable in the Lyapunov sense as solutions to the considered nonlinear systems.
\end{remark}

The proof of Theorem \ref{thm:Turing} is presented at the end of this subsection and is a direct application of Theorem \ref{thm:stab:FRD} combined with the following result.

\begin{theorem}[Linear Turing instability]\label{thm:linear:Turing}
    Let $\alpha\in (0,1)$. 
    Consider the system 
    \begin{equation} \label{FRD:lin}
    \begin{split}
        \partial_t^\alpha u&= D_1 \Delta u +au+bv, \\
         \partial_t^\alpha v&=  D_2\Delta v +cu+dv,
         \end{split}
         \qquad x\in \Omega, \quad t>0,
    \end{equation}
    $a,b,c,d\in \R$, on a bounded domain $\Omega\subset\R^n$ and supplemented either with the Dirichlet boundary condition \eqref{Dirichlet}
    or the Neumann boundary condition \eqref{Neumann}.
    We assume 
    that the matrix 
        \begin{equation}\label{matrix:A}
            A= 
            \left(
                \begin{array}{cc}
                     a& b \\
                     c& d
                \end{array}
            \right)
        \end{equation}
     has two eigenvalues $\lambda_\pm\in \C$
     such that 
        \begin{equation} \label{lambda:pm}
            |\arg (\lambda_\pm)|>  \frac{\alpha\pi}{2}.
        \end{equation}
    Assume also that $a>0$.
    Then, for sufficiently small $D_1>0$, the zero stationary solution of the initial-boundary value problem for system \eqref{FRD:lin} is unstable in the sense 
    that the operator defined by the right-hand side of system \eqref{FRD:lin} has an eigenvalue satisfying $|\arg (\lambda)|<\alpha\pi/2$.
\end{theorem}

However, we first prove a simple result concerning the roots of quadratic equations. 

\begin{lemma}\label{lem:A2}
For the matrix $A$ defined in \eqref{matrix:A},
we denote its trace by  
$\T = a+c$ and its determinant by $\D = ad-bc$.
Let $\alpha\in (0,1)$. The matrix $A$  has both eigenvalues $\lambda_\pm\in \C$
satisfying
    \begin{equation} \label{lambda:pm:2}
        |\arg (\lambda_\pm)|>  \frac{\alpha\pi}{2}
    \end{equation}
  if and only if $\D>0$, and  either 
    \begin{equation}\label{t:1:1}
         \T<0,
    \end{equation}
or
    \begin{equation}\label{t:2:2}
        \T\ge 0, \quad\text{and} \quad  \frac{\T}{2\sqrt{\D}} < \cos\left(\frac{\alpha \pi}{2}\right).
    \end{equation}
\end{lemma}

\begin{proof}
The characteristic polynomial of the matrix $A$ is given by   
$w(\lambda)=\lambda^2-\lambda\T+\D$
and its roots (the eigenvalues of the matrix $A$) are given by formula 
    \begin{equation}\label{eigenvalues:0}
        \lambda_\pm = \frac{\T\pm \sqrt{\T^2-4\D}}{2}.
    \end{equation}
If $\D< 0$, then $\lambda_+$ is a real and positive eigenvalue, which cannot satisfy inequality \eqref{lambda:pm:2}.
The case $\D=\lambda_-\lambda_+=0$ is also excluded by inequality \eqref{lambda:pm:2}.
Now, for $\D >0$, one of the three possibilities can occur.
    \begin{enumerate}
        \item If $\T < 0$, then both eigenvalues $\lambda_\pm$ have negative real parts and they both satisfy inequality \eqref{lambda:pm:2}.

        \item If $\T \ge 0$ and $4\D\le \T^2$, then both eigenvalues are real and  positive, so, they do not satisfy \eqref{lambda:pm:2}.

        \item If $\T \ge 0$ and $4\D> \T^2$, then $\lambda_\pm$ are pairs of conjugate complex numbers with positive real parts. In this case, we complete the proof by using the relation 
             \begin{equation}
                \cos\big(\left|\arg \left(\lambda_\pm\right)\right|\big)= \frac{\Re \lambda_\pm}{\|\lambda_\pm\|} =
                \frac{\T}{2\sqrt D},
            \end{equation}
         and the fact that cosine is a decreasing function on $[0, \pi/2]$.
    \end{enumerate}
\end{proof}

\begin{proof}[Proof of Theorem \ref{thm:linear:Turing}.]
    We consider
an
orthonormal basis of eigenfunctions $\{w^k\}_{k=1}^\infty $ 
associated with positive eigenvalues $\{\lambda_k\}_{k=1}^\infty$
    of the Laplace operator, either with the Dirichlet or the 
Neumann condition.
Thus, a solution of system \eqref{FRD:lin} has the form 
$$
u(t)=\sum_{k=1}^\infty p_k(t) w^k, \quad v(t)=\sum_{k=1}^\infty q_k(t) w^k,
$$
with suitable time-dependent coefficients $p_k(t), q_k(t)$ satisfying the systems of fractional differential equations 
    \begin{equation}
        \begin{split}\label{pk:qk:system}
            \partial_t^\alpha p_k&=-D_1 \lambda_kp_k+a  p_k+bq_k,\\
            \partial_t^\alpha q_k&=- D_2 \lambda_k q_k+cp_k+dq_k.
        \end{split}
    \end{equation}
We look for a solution of system \eqref{pk:qk:system} in the form 
    $$
        p_k(t) =\bar p_k E_\alpha(\lambda t^\alpha), 
        \qquad 
        q_k(t) =\bar q_k E_\alpha(\lambda t^\alpha),
    $$
with the Mittag-Leffler function \eqref{Mittag-Leffler} and unknown numbers $\lambda$, $\bar p_k$,  $q_k$ which, by equations \eqref{eq: FDE simple:0}-\eqref{eq: sol D^a u=ru:0},
satisfy the linear system
    \begin{equation}\label{pk:qk:linear:system}
        \begin{split}
            \lambda \bar p_k&=(a-\su \lambda_k) \bar p_k +b\bar q_k,\\
             \lambda \bar q_k&= c\bar p_k+(d-\sv\lambda_k)\bar q_k.
        \end{split}
    \end{equation}
In the following, we analyze eigenvalues $\lambda$ of system \eqref{pk:qk:linear:system} by applying Lemma \ref{lem:A2} to the matrix
\begin{equation}\label{matrix:Ak}
    A_k= \left(
        \begin{array}{cc}
             a-D_1 \lambda_k& b \\
             c& d-D_2\lambda_k
        \end{array}
    \right)
\end{equation}
 with the trace and determinant 
    \begin{equation}
        \begin{split}
            \T_k &= a-\su\lambda_k+c-\sv\lambda_k=\T -\lambda_k(\su+\sv),\\
            \D_k &= (a-\su\lambda_k)(d-\sv\lambda_k)-bc =\D -\lambda_k (\su d+\sv a) +\su\sv \lambda_k^2
            .
        \end{split}
    \end{equation}
Our goal is to  find conditions on the matrix $A_k$, such that it has an eigenvalue 
$\lambda \in \C \setminus \{0\}$
satisfying 
\begin{equation}\label{lambda:Turing}
|\arg(\lambda)|<\alpha\pi/2,
\end{equation}
given that 
the matrix $A$ in \eqref{matrix:A} has both eigenvalues $\lambda_\pm\in \C$
 satisfying 
 $$
 |\arg (\lambda_\pm)|>  \frac{\alpha\pi}{2}.
  $$

First, notice that if $\T <0$, then $\T_k <0$. Moreover, if both $a<0$ and $d<0$ 
then $\D>0$ implies $\D_k>0$. Thus, for $a<0$ and $d<0$, by Lemma \ref{lem:A2}, both eigenvalues of the matrix $A_k$ have  negative real parts and inequality \eqref{lambda:Turing} cannot be satisfied. Consequently, in what follows, we assume that $a>0$ or $d>0$ (or both).

Form now on, without loss of generality, we assume that $a>0$. 
 Since $\T_k<0$,  our goal is to achieve 
  \begin{equation}
        \D_k = \D -\lambda_k (\su d+\sv a) +\su\sv \lambda_k^2<0 \quad\text{for some} \quad k\in\N.
    \end{equation}
This inequality is satisfied by some $\lambda_k>0$ if the minimum of the corresponding parabola 
is attained at a positive value and 
is negative. This occurs if
    \begin{equation}
      \frac{D_1 d+D_2 a}{2D_1 D_2 }> 0 \qquad \text{and} \qquad   -\frac{(\su d+\sv a)^2 - 4D_1 D_2  \D}{4D_1 D_2 } <0.
    \end{equation}
Denoting $\theta = D_1 /D_2 $, these inequalities become
    \begin{equation}
      \frac{d}{2D_2 }+\frac{a}{2D_1 }>0 \qquad \text{and} \qquad   {\theta d^2+ \frac{1}{\theta} a^2 }  > 4\left(\D - \frac{a d}{2}\right).
    \end{equation}
Now, it is clear that these inequalities will hold if, for fixed $D_2 >0$ and for $a>0$, we choose $D_1 >0$ sufficiently small. Consequently, there exist numbers $0<\Lambda_-<\Lambda_+$ (calculated,  {\it e.g.},~in the monograph \cite[Theorem 7.1]{P15}) such that
for $\lambda_k\in [\Lambda_-,\Lambda_+]$, the matrix $A_k$ in \eqref{matrix:Ak} has a real and positive eigenvalue (so satisfying inequality  \eqref{lambda:Turing}).
\end{proof}

\begin{remark}In the proof of Theorem \ref{thm:linear:Turing}, we have skipped the analysis of the case when the coefficients of the matrix $A$ in \eqref{matrix:A} satisfy
    \begin{equation}\label{t:2:3}
        \D>0, \quad \T\ge 0, \quad \frac{\T}{2\sqrt{\D}} < \cos\left(\frac{\alpha \pi}{2}\right),
    \end{equation}
and where it is possible to show that for $a>0$ and for sufficiently small $D_1 >0$, we have
        \begin{equation}
            \frac{\T_k}{2\sqrt{\D_k}} > \cos\left(\frac{\alpha \pi}{2}\right),
        \end{equation}
        for some $k$.
        We postpone detailed calculations to our forthcoming paper. 
\end{remark}

\begin{proof}[Proof of Theorem \ref{thm:Turing}]
   Apply the usual linearization procedure, Theorem \ref{thm:linear:Turing} with the matrix $A$ defined in \eqref{matrix:A:fg}, and Theorem \ref{thm:stab:FRD}, Part 2.
\end{proof}


\subsection{Non-constant stationary solutions}

In the study of stability of non-constant stationary solutions of the fractional reaction-diffusion system \eqref{eq:FRD}, the linearization principle produces a linear system of fractional reaction-diffusion equations analogous to \eqref{eq:FRD:constant}-\eqref{eq:lin:constant}, but with the matrix $A$ dependent on $x$. 
Here, the methods employed in the proof of Theorem~\ref{thm:stab:FRD} are not directly applicable because the linearized system generally does not admit an explicit solution in terms of the eigenfunctions of $-\Delta$. However, we can still use Corollary \ref{cof:negative spectal bound} to establish stability estimates for the resolvent families required by Theorem \ref{thm:stab}, provided that the spectral bound of the linearized operator is strictly negative.

As a final application of our general results, we extend to the fractional setting the classical result concerning the instability of non-constant stationary solutions to the general fractional reaction-diffusion equation subject to the Neumann boundary condition
    \begin{equation}\label{eq: FRDS 1D}
        \begin{split}
            \partial^\alpha_t v&= \Delta v + g(v),  \quad x\in \Omega, \quad t>0,\\
            \nu\cdot  \nabla v &=0,     \qquad\qquad \quad  x \in \partial\Omega, \quad t>0,
        \end{split}
    \end{equation}
with $\alpha\in (0,1)$, with an arbitrary $C^2$-nonlinearity $g=g(v)$, considered on a bounded domain $\Omega\subset\R^n$.

In the classical case where $\alpha=1$, the following theorem was initially proved by Chafee \cite{C75} in the one-dimensional setting, and later generalized to higher dimensions by Casten and Holland \cite{CH78} and Matano \cite{M79}. Applying the general instability tools from Section \ref{sec:instability}, we immediately obtain the analogous result in the fractional case.

\begin{theorem} \label{thm:non:const}
    Let $\alpha\in (0,1)$. Assume that $\Omega\subset\mathbb{R}^n$ is bounded, convex
    and $g=g(v)$ has an arbitrary $C^2$-nonlinearity.
    Then, all non-constant stationary solutions of problem \eqref{eq: FRDS 1D} are unstable.
\end{theorem}

\begin{proof}
    Let $V=V(x)$ be a non-constant stationary solution of problem \eqref{eq: FRDS 1D}. Introducing the new variable 
    $u=v-V$ we can rewrite problem \eqref{eq: FRDS 1D} in the form
     \begin{align}\label{eq: FRDS 1D:lin}
                     &\partial^\alpha_t u= \Delta u + g'(V)u + f(u), & x\in \Omega, \quad t>0,\\
                   &\nu\cdot  \nabla u =0, & x \in \partial\Omega, \quad t>0,
            \end{align}
    where $f=f(u)$ is a suitable Taylor reminder satisfying assumption \eqref{f:flat:instab} with $\eta=1$.
    It is proved in References \cite{CH78,M79} by using a variational argument that the operator 
        $$\A u=\Delta u + g'(V)u,$$ 
    subject to the Neumann boundary condition has a positive eigenvalue when $V(x)$ is a non-constant stationary solution. Since this is a symmetric operator, its spectrum is real. Thus, by Proposition \ref{prop:instab}, the operator $\A$ satisfies Assumption \ref{ass:instab} 
    and Theorem~\ref{thm:instab} implies the instability of the zero solution of problem \eqref{eq: FRDS 1D:lin}.
\end{proof}


\end{document}